\documentclass{article}
\usepackage{amsmath}
\usepackage{amsfonts}
\usepackage{amsthm}
\usepackage{amssymb}
\usepackage{enumerate}

\newtheorem{thm}{Theorem}[section]
\newtheorem{prop}[thm]{Proposition}
\newtheorem{cor}[thm]{Corollary}
\newtheorem{lem}[thm]{Lemma}

\newtheorem*{thmnn}{Theorem}

\theoremstyle{definition}
\newtheorem{defn}[thm]{Definition}
\newtheorem{rmk}[thm]{Remark}

\begin{document}

\title{Generalized Hyperbolic Spaces Associated with Arbitrary Quadratic Forms}

\author{Shaul Zemel}

\maketitle


\section*{Introduction}

Clifford algebras were observed, over a century ago, to play a very important role in the theory of quadratic spaces. Their group of units contain two natural subgroups, the Clifford group acting on the quadratic space itself, and the paravector Clifford group, acting on the extension of the space by the scalars as an extended quadratic space. The classical theory involved real quadratic spaces (and sometimes complex ones), initially in the context of the Newtonian physics of movements in space, but as the understanding of the algebraic theory behind these notions became deeper, the natural definitions over more general fields became abundant. Allowing the base field to become a more general commutative ring led to interesting additional features---see \cite{[Ba]}, \cite{[Mc]}, and \cite{[Z]} among the vast literature on the subject and the many different approaches to it.

\smallskip

The simple observation that the Clifford algebra of the hyperbolic plane is a matrix algebra has the consequence that the Clifford algebra associated with the direct sum of a quadratic space $(V,q)$ with a hyperbolic plane is the matrix algebra over the Clifford algebra $(V,q)$ itself. This lies in the heart of Vahlen's construction of the groups that are now named after him, consisting of certain $2\times2$ matrices over a Clifford algebra. For particular real spaces, Vahlen could show that his group operates on an appropriate space (now known to be a real hyperbolic space) via M\"{o}bius transformations, a fact that was established in more details and more generality in \cite{[Ah]}. This generalizes the well-known fact that the 2-dimensional and 3-dimensional hyperbolic spaces are the symmetric spaces for $\operatorname{SL}_{2}(\mathbb{R})$ and $\operatorname{SL}_{2}(\mathbb{C})$ respectively, identifying these groups as the spin groups of respective signatures $(2,1)$ and $(3,1)$. There is also a paravector analogue, considered, for example, in \cite{[Ma]}.

Now, \cite{[EGM]} extended the definition of the (paravector) Vahlen group to any quadratic space over any field, and established the equivalent conditions for defining them in case the vector space is strongly anisotropic, a technical condition that is required for the only elements of the Clifford algebra whose natural action preserves paravectors are those of the paravector Clifford group and 0. The paper \cite{[Mc]} found the exact definition for the entries of the Vahlen groups, and showed that the proofs from \cite{[EGM]} work both for the usual and the paravector Vahlen groups, for quadratic modules, non-degenerate as well as degenerate, over much more general rings. The Vahlen group is then isomorphic to the scalar norm elements of the Clifford group of the direct sum of $(V,q)$ with a hyperbolic plane, and the paravector one is isomorphic to the scalar norm elements of the even Clifford group of a slightly larger quadratic space. The structure of the Clifford algebras in the degenerate case (over fields) is briefly investigated in \cite{[Ab]}, and more details are given in the predecessor \cite{[Z]} of the current paper.

\smallskip

The paper \cite{[EGM]} presents the action of the real Vahlen group (which in our terminology is the special paravector Vahlen group) of a definite real quadratic space on the hyperbolic space, with the three different models of that space: The hyperboloid model, the half-space model, and the ball model. The goal of the current paper is to extend this action to any Vahlen group, on both the hyperboloid model and the half-space model. The former, which is an orthogonal action, is rather straightforward. However, the action on the half-space model, using M\"{o}bius transformations, is more delicate, and requires a certain completion at infinity for working properly. The ball model requires additional assumptions, and we do not consider it in this paper. More precisely, we establish the following results.
\begin{thmnn}
To a quadratic space $(V,q)$ over a field $\mathbb{F}$ of characteristic different from 2, with the choice of a scalar $c$ from $\mathbb{F}$, one attaches a space $\mathbf{H}_{V,q}^{c}$, on which the Vahlen group operates transitively via M\"{o}bius transformations. This is a model for the action of the Vahlen group, via its image as an orthogonal group, on the set of vectors of quadratic value $c$ in the direct sum of $(V,q)$ and a hyperbolic plane, with the set of vectors that are perpendicular to the entire space excluded when $c=0$.
\end{thmnn}

\begin{thmnn}
Given $(V,q)$ over $\mathbb{F}$ and $c$ as above, there is a space $\widetilde{\mathbf{H}}_{V,q}^{c}$ with a transitive action of the paravector Vahlen group, that is also described in terms of M\"{o}bius transformations. The linear model is based on the set of vectors of norm $c$ is a slightly larger quadratic space, where again for $c=0$ the vectors that are orthogonal to every vector in the space are excluded.
\end{thmnn}
For the exact formulations, see Theorems \ref{thmVqc} and \ref{thmpara} below. The results for the two types of Vahlen groups are unrelated but goes along the exact same lines, but some delicate details are different, and we give the detailed proofs for both cases.

\smallskip

The paper is divided into three sections. Section \ref{VahGrps} skims through the basic definitions of Clifford algebras and Vahlen groups that are required for stating and proving our results. Section \ref{ActVah} presents the results for the ordinary Vahlen group, and Section \ref{ActPara} carries out the same for the paravector Vahlen groups.

\section{Clifford Algebras and Vahlen Groups \label{VahGrps}}

Let $\mathbb{F}$ be a field of characteristic different from 2, let $V$ be a finite-dimensional vector space over $\mathbb{F}$, and let $q$ be a quadratic form on $V$. We denote the associated symmetric bilinear form by $(u,v):=q(u+v)-q(u)-q(v)$ as usual. We do not assume that $q$ non-degenerate, and we denote the kernel of the associated map from $V$ to its dual by
\begin{equation}
V^{\perp}:=\{v \in V|(u,v)=0\ \forall u \in V\},\qquad\mathrm{as\ well\ as\ set}\qquad\overline{V}:=V/V^{\perp}. \label{Vperp}
\end{equation}
Since the characteristic of $\mathbb{F}$ is different from 2, and $(v,v)=2q(v)$ for every $v \in V$, we deduce that $q$ vanishes identically on $V^{\perp}$, and it thus factors through a natural non-degenerate quadratic form $\overline{q}$ on $\overline{V}$. We denote by $\mathcal{C}$ the Clifford algebra $\mathcal{C}(V,q)$, and we identify $\mathbb{F}$ and $V$ with their images in $\mathcal{C}$. This is an $\mathbb{F}$-algebra that is generated by $V$ and is determined by the condition that the square $v^{2}$ of $v \in V\subseteq\mathcal{C}$ equals $q(v)\in\mathbb{F}$, and as a consequence we get the equality $uv+vu=(u,v)$ for every $u$ and $v$ in $V$.

Recall that $\mathcal{C}$ is graded, i.e., it decomposes as the sum of the even part $\mathcal{C}_{+}$ and the odd part $\mathcal{C}_{-}$, and that it comes equipped with an involution $\alpha\mapsto\alpha'$, which is the identity on $\mathcal{C}_{+}$ and multiplies elements of $\mathcal{C}_{-}$ by $-1$. It also carries an anti-involution, called \emph{transposition} and denoted by $\alpha\mapsto\alpha^{*}$, which leaves $\mathbb{F}$ and $V$ invariant (but inverts the order of multiplications). These two involutions commute, and their composition (in either order) produces the \emph{Clifford involution}, which is denoted by $\alpha\mapsto\overline{\alpha}$. Using the grading involution, and in the spirit of Equation \eqref{parabil} from the paravector case considered below, the pairing formula becomes
\begin{equation}
uv'+vu'=u'v+v'u=-(u,v)\in\mathbb{F}\subseteq\mathcal{C}\mathrm{\ for\ every\ }u\mathrm{\ and\ }v\mathrm{\ in\ }V. \label{pairC}
\end{equation}

\smallskip

We define the \emph{twisted center} of $\mathcal{C}$ (named so after, e.g., \cite{[Mc]}) and the \emph{Clifford group} to be
\begin{equation}
\widetilde{Z}(\mathcal{C}):=\{\alpha\in\mathcal{C}|\alpha v=v\alpha'\ \forall v \in M\}\quad\mathrm{and}\quad\Gamma(M,q):=\{\alpha\in\mathcal{C}^{\times}|\alpha V\alpha'^{-1}=V\} \label{twZClgrp}
\end{equation}
respectively. Note that some authors require only inclusion in the definition of the Clifford group, which over rings might make a difference (see \cite{[Z]}), but since we work over a field, the injectivity of conjugation inside $\mathcal{C}$ implies, via dimension consideration, that inclusion and equality are equivalent there. Theorem 1.12 of \cite{[Z]} determines the former algebra to be the image of $\mathcal{C}(V^{\perp},0)=\bigwedge^{*}V^{\perp}$ inside $\mathcal{C}$, and its group of invertible elements is thus $\mathbb{F}^{\times}\oplus\bigoplus_{r>0}\bigwedge^{r}V^{\perp}$. Note that orthogonal maps on $V$, i.e., maps $\varphi:V \to V$ satisfying $q\circ\varphi=q$ need not be injective, as elements of $V^{\perp}$ can lie in kernels of such maps, and that they take $V^{\perp}$ into itself. We thus define the orthogonal group $\operatorname{O}(V,q)$ to consist of invertible orthogonal maps, and, following \cite{[Z]}, we define the subgroup
\begin{equation}
\operatorname{O}_{V^{\perp}}(V,q):=\big\{\varphi\in\operatorname{O}(V,q)\big|\varphi|_{V^{\perp}}=\operatorname{Id}_{V^{\perp}}\big\}=\{\varphi\in\operatorname{O}(V,q)|\varphi(v)=v\ \forall v \in V^{\perp}\}. \label{orthgrp}
\end{equation}
When $(V,q)$ is non-degenerate, the Cartan--Dieudonn\'{e} Theorem states that $\operatorname{O}_{V^{\perp}}(V,q)=\operatorname{O}(V,q)$ is generated by \emph{reflections}, i.e., maps of the form $r_{v}$ taking $u \in V$ to $u-\frac{(u,v)}{q(v)}v$, for $v \in V$ with $q(v)\neq0$ (for the proof see Corollary 4.3 of \cite{[MH]} or Subsection 43B of \cite{[O]}). Corollary 3.6 of \cite{[Z]} extends this statement (for $\operatorname{O}_{V^{\perp}}(V,q)$) to the degenerate case.

Lemma 2.2 of \cite{[Z]} shows that $\operatorname{O}_{V^{\perp}}(V,q)$ lies in a short exact sequence
\begin{equation}
0\to\big(\operatorname{Hom}_{\mathbb{F}}(\overline{V},V^{\perp}),+\big)\to\operatorname{O}_{V^{\perp}}(V,q)\to\operatorname{O}(\overline{V},\overline{q})\to1 \label{OSES}
\end{equation}
(in fact, Lemma 2.1 of that reference puts $\operatorname{O}(V,q)$ inside a similar exact sequence, with $\operatorname{O}(\overline{V},\overline{q})$ multiplied by $\operatorname{GL}(V^{\perp})$), and Corollary 3.6, Proposition 3.7, and Theorem 3.8 of that reference yield the following result.
\begin{thm}
The Clifford group $\Gamma(M,q)$ from Equation \eqref{twZClgrp} lies in the short exact sequence \[1\to\textstyle{\mathbb{F}^{\times}\oplus\bigoplus_{r>0}\bigwedge^{r}V^{\perp}}\to\Gamma(V,q)\stackrel{\pi}{\to}\operatorname{O}_{V^{\perp}}(V,q)\to1,\] where the group on the left is $\widetilde{Z}(\mathcal{C})^{\times}$. Moreover, if $\alpha$ is in $\Gamma(V,q)$ then so are $\alpha'$, $\alpha^{*}$, and $\overline{\alpha}$, with $\pi(\alpha')=\pi(\alpha)$ and $\pi(\alpha^{*})=\pi(\overline{\alpha})=\pi(\alpha)^{-1}$. \label{GammaO}
\end{thm}
We shall also be needing the graded subgroup of the Clifford group. First, the fact that $\operatorname{O}_{V^{\perp}}(V,q)$ lies in the short exact sequence from Equation \eqref{OSES}, in which it surjects onto a classical, reductive, non-degenerate orthogonal group with a unipotent kernel, shows that its elements have determinant $\pm1$. We denote, as usual, the subgroup defined by the determinant 1 condition by $\operatorname{SO}_{V^{\perp}}(V,q)$. Moreover, recall that the surjectivity of the map from Theorem \ref{GammaO} is based on generation by reflections, and reflections are the images of vectors from $V$ (by, e.g., Proposition 3.4 of \cite{[Z]}, among earlier results), thus elements of $\mathcal{C}_{-}$. We therefore define $\Gamma_{+}(V,q):=\Gamma(V,q)\cap\mathcal{C}_{+}$, $\Gamma_{-}(V,q):=\Gamma(V,q)\cap\mathcal{C}_{-}$, and $\Gamma_{\pm}(V,q):=\Gamma_{+}(V,q)\cup\Gamma_{-}(V,q)$, for which we obtain the following simple consequence, which resembles Corollary 3.9 of \cite{[Z]}.
\begin{cor}
The subset $\Gamma_{\pm}(V,q)$ is a subgroup of $\Gamma(V,q)$, which also surjects onto $\operatorname{O}_{V^{\perp}}(V,q)$ via $\pi$, with kernel $\mathbb{F}^{\times}\oplus\bigoplus_{s>0}\bigwedge^{2s}V^{\perp}$, which equals $\widetilde{Z}(\mathcal{C})_{+}^{\times}$, the group of units in $\widetilde{Z}(\mathcal{C})_{+}=\widetilde{Z}(\mathcal{C})\cap\mathcal{C}_{+}$. The subset $\Gamma_{+}(V,q)$ is the inverse image of $\operatorname{O}_{V^{\perp}}(V,q)$ under this restricted projection, and is thus a subgroup of index 2 there, and $\Gamma_{-}(V,q)$ is the non-trivial coset of $\Gamma_{+}(V,q)$ inside $\Gamma_{\pm}(V,q)$. \label{grClgrp}
\end{cor}
We recall that in the classical case, of non-degenerate $(V,q)$, the Clifford group from Equation \eqref{twZClgrp} factors as the union of graded pieces, as in Corollary \ref{grClgrp}. This is so, because the twisted center reduces to $\mathbb{F}$ in this case, which is contained in $\mathcal{C}_{+}$, and thus the subgroup from that corollary is the entire Clifford group. This is, of course, not the case in the degenerate case.

\smallskip

We recall from \cite{[Ma]}, \cite{[Mc]}, and others that the space $\mathbb{F} \oplus V$ inside $\mathcal{C}$, called the space of \emph{paravectors}, comes with the quadratic form $q_{\mathbb{F}}$ that takes an element $\xi=a+v\in\mathbb{F} \oplus V$, with $a\in\mathbb{F}$ and $v \in V$, to $q(v)-a^{2}$. Moreover, for any $\xi\in\mathbb{F} \oplus V$ we can express $q_{\mathbb{F}}(\xi)$ as $-\xi\xi'$ (or equivalently $-\xi\overline{\xi}$), and then for another element $\eta=b+u\in\mathbb{F} \oplus V$, we obtain that
\begin{equation}
\xi\eta'+\eta\xi'=\xi'\eta+\eta'\xi=-(\xi,\eta)_{\mathbb{F}},\qquad\mathrm{where}\qquad(\xi,\eta)_{\mathbb{F}}:=(u,v)-2ab \label{parabil}
\end{equation}
is the symmetric bilinear form induced from $q_{\mathbb{F}}$ on $\mathbb{F} \oplus V$. We denote the resulting quadratic space by $V_{\mathbb{F}}$, and, similarly to Equation \eqref{twZClgrp}, we define the \emph{paravector Clifford group} to be
\begin{equation}
\widetilde{\Gamma}(M,q):=\{\alpha\in\mathcal{C}^{\times}|\alpha(\mathbb{F} \oplus V)\alpha'^{-1}=\mathbb{F} \oplus V\}. \label{paragrp}
\end{equation}
Also here one can restrict the requirement to be inclusion only, which is equivalence in our setting, working over a field.

Note that the space $V_{\mathbb{F}}^{\perp}$ as defined in Equation \eqref{Vperp} for $V_{\mathbb{F}}$ is the image of $V^{\perp} \subseteq V$ (this is because the characteristic of $\mathbb{F}$ is not 2), and, being in a short exact sequence as in Equation \eqref{OSES}, the group $\operatorname{O}_{V^{\perp}}(V_{\mathbb{F}},q_{\mathbb{F}})$ has a subgroup $\operatorname{SO}_{V^{\perp}}(V_{\mathbb{F}},q_{\mathbb{F}})$ of index 2 defined by the determinant 1 condition. Note that $\xi\mapsto-\xi'$ is the reflection $r_{1}$ in the element 1 of $V_{\mathbb{F}}$, representing the non-trivial coset of $\operatorname{SO}_{V^{\perp}}(V_{\mathbb{F}},q_{\mathbb{F}})$ inside $\operatorname{O}_{V^{\perp}}(V_{\mathbb{F}},q_{\mathbb{F}})$ (and conjugation by it preserves this normal subgroup). Proposition 3.16 and Theorem 3.17 of \cite{[Z]} then yield, via Corollary 3.6 of that reference again, the following result.
\begin{thm}
The paravector Clifford group $\widetilde{\Gamma}(M,q)$ from Equation \eqref{paragrp} sits in the short exact sequence \[1\to\textstyle{\mathbb{F}^{\times}\oplus\bigoplus_{s>0}\bigwedge^{2s}V^{\perp}}\to\widetilde{\Gamma}(V,q)\stackrel{\widetilde{\pi}}{\to}\operatorname{SO}_{V^{\perp}}(V_{\mathbb{F}},q_{\mathbb{F}})\to1.\] The kernel here is the same one from Corollary \ref{grClgrp}. In addition, for $\alpha\in\widetilde{\Gamma}(V,q)$, also its images $\alpha'$, $\alpha^{*}$, and $\overline{\alpha}$ lie in this group, and their $\widetilde{\pi}$ images are $r_{1}\widetilde{\pi}(\alpha)r_{1}$, $r_{1}\widetilde{\pi}(\alpha)^{-1}r_{1}$ , and $\widetilde{\pi}(\alpha)^{-1}$ respectively. \label{tildeSES}
\end{thm}

\smallskip

We shall need particular subgroups of both the usual and the paravector Clifford group.
\begin{lem}
The norm map, defined by $N(\alpha):=\alpha\overline{\alpha}$, takes $\Gamma(V,q)$ into $\widetilde{Z}(\mathcal{C})^{\times}$ and $\widetilde{\Gamma}(M,q)$ into $\widetilde{Z}(\mathcal{C})_{+}^{\times}$. The subsets $\Gamma^{\mathbb{F}^{\times}}(V,q):=\{\alpha\in\Gamma(V,q)|N(\alpha)\in\mathbb{F}^{\times}\}$ and $\widetilde{\Gamma}^{\mathbb{F}^{\times}}(V,q):=\{\alpha\in\widetilde{\Gamma}(V,q)|N(\alpha)\in\mathbb{F}^{\times}\}$ are subgroups of $\Gamma(V,q)$ and $\widetilde{\Gamma}(M,q)$ respectively, that surject onto $\operatorname{O}_{V^{\perp}}(V,q)$ and $\operatorname{SO}_{V^{\perp}}(V_{\mathbb{F}},q_{\mathbb{F}})$. \label{NinFx}
\end{lem}

\begin{proof}
We have seen that for $\alpha$ in $\Gamma(V,q)$ (resp. $\widetilde{Z}(\mathcal{C})^{\times}$), the element $\overline{\alpha}$ also lies in that group, and its image under $\pi$ (resp. $\widetilde{\pi}$) is the inverse of that of $\alpha$. Therefore $N(\alpha)$ lies in the kernel of this map, which is then determined in Theorem \ref{GammaO} (resp. Theorem \ref{tildeSES}), proving the first assertion. Now, the centrality of $\mathbb{F}$ inside $\mathcal{C}$ and the Clifford involution inverting multiplication orders imply that if $\beta\in\mathcal{C}$ satisfies $N(\beta)\in\mathbb{F}$ then for any other element $\alpha\in\mathcal{C}$ we have $N(\alpha\beta)=\alpha\beta\overline{\beta}\overline{\alpha}=\alpha N(\beta)\overline{\alpha}=N(\alpha)N(\beta)$. Since $N(\beta^{-1})=\beta^{-1}\overline{\beta}^{-1}$ is a conjugate of $N(\beta)^{-1}=\overline{\beta}^{-1}\beta^{-1}$, they are equal in case $N(\beta)\in\mathbb{F}^{\times}$, which together with the product formula implies that $\Gamma^{\mathbb{F}^{\times}}(V,q)$ and $\widetilde{\Gamma}^{\mathbb{F}^{\times}}(V,q)$ are closed under the group operations as desired. Finally, Remark 2.9 and Corollary 3.6 of \cite{[Z]} show (among earlier references) that $\operatorname{O}_{V^{\perp}}(V,q)$ is generated by reflections $r_{v}$ for $v \in V$ with $q(v)\neq0$, and combining it with Lemma 3.12 of that reference, we deduce that $\operatorname{SO}_{V^{\perp}}(V_{\mathbb{F}},q_{\mathbb{F}})$ is generated by compositions $r_{\xi} \circ r_{1}$ for $\xi \in V_{\mathbb{F}}$ with $q_{\mathbb{F}}(\xi)\neq0$. Since Proposition 3.4 of \cite{[Z]} shows that the former map is the $\pi$-images of $v\in\Gamma(V,q)$, and $N(v)=-q(v)\in\mathbb{F}^{\times}$, and the latter one is obtained via $\widetilde{\pi}$ from $\xi\in\widetilde{\Gamma}(M,q)$, for which $N(\xi)=-q_{\mathbb{F}}(\xi)\in\mathbb{F}^{\times}$, we deduce that the generators of the groups in question are indeed in the image of our subgroups. This proves the lemma.
\end{proof}
In fact, the kernel $\widetilde{Z}(\mathcal{C})_{+}^{\times}$ of $\widetilde{\pi}$ in Theorem \ref{tildeSES} coincides with $Z(\mathcal{C})_{+}^{\times}$, and is thus central. Thus the proof of Lemma \ref{NinFx} shows that for $\widetilde{\Gamma}(M,q)$, the norm map is a homomorphism into $\widetilde{Z}(\mathcal{C})_{+}^{\times}$, and thus $\widetilde{\Gamma}^{\mathbb{F}^{\times}}(V,q)$ is a subgroup since it is the inverse image of a subgroup under a group homomorphism. But since this argument fails for $\Gamma(M,q)$ (since $\widetilde{Z}(\mathcal{C})^{\times}$ is no longer central in general), a different proof is required in this case, and we gave a unified proof.

The proofs of Lemma \ref{NinFx} and Corollary \ref{grClgrp} combine to produce the following corollary.
\begin{cor}
The two intersections $\Gamma^{\mathbb{F}^{\times}}_{\pm}(V,q)=\Gamma^{\mathbb{F}^{\times}}(V,q)\cap\Gamma_{\pm}(V,q)$ and $\Gamma^{\mathbb{F}^{\times}}_{+}(V,q)=\Gamma^{\mathbb{F}^{\times}}(V,q)\cap\Gamma_{+}(V,q)$ are subgroups of $\Gamma(V,q)$, the restriction of $\pi$ to which is surjective onto $\operatorname{O}_{V^{\perp}}(V,q)$ and $\operatorname{SO}_{V^{\perp}}(V,q)$ respectively. The intersection $\Gamma^{\mathbb{F}^{\times}}_{-}(V,q)=\Gamma^{\mathbb{F}^{\times}}(V,q)\cap\Gamma_{-}(V,q)$ is the non-trivial coset of the latter group inside the former. \label{Normgrade}
\end{cor}
Note that elements mapping to reflections in the paravector Clifford group $\widetilde{\Gamma}(M,q)$ from Equation \eqref{paragrp} are no longer graded in general, so that subgroups of $\widetilde{\Gamma}(M,q)$ of the form considered in Corollaries \ref{grClgrp} and \ref{Normgrade} are less natural. In fact, Proposition 3.19 of \cite{[Z]} shows that the graded elements of $\widetilde{\Gamma}(M,q)$ are precisely those that are also in $\Gamma(V,q)$, and are thus in $\Gamma_{\pm}(V,q)$ (such elements $\alpha$ also satisfy $\alpha\alpha'=\pm1$, the sign being that of the determinant of $\pi(\alpha)$ for $\widetilde{\pi}(\alpha)$ to be in $\operatorname{SO}_{V^{\perp}}(V_{\mathbb{F}},q_{\mathbb{F}})$).

As another consequence of Lemma \ref{NinFx}, combined with Corollary \ref{Normgrade}, we obtain a smaller subgroup of each of $\Gamma(V,q)$ and $\widetilde{\Gamma}(M,q)$.
\begin{cor}
The subsets $\Gamma^{1}(V,q):=\{\alpha\in\Gamma(V,q)|N(\alpha)=1\}$ of $\Gamma(V,q)$ and $\widetilde{\Gamma}^{1}(V,q):=\{\alpha\in\widetilde{\Gamma}(V,q)|N(\alpha)=1\}$ of $\widetilde{\Gamma}(M,q)$ are subgroups, that are normal subgroups of $\Gamma^{\mathbb{F}^{\times}}(V,q)$ and $\widetilde{\Gamma}^{\mathbb{F}^{\times}}(V,q)$ from Lemma \ref{NinFx} respectively. The same applies to $\Gamma^{1}_{\pm}(V,q):=\Gamma^{1}(V,q)\cap\Gamma_{\pm}(V,q)$ inside $\Gamma_{\pm}(V,q)$ and to its index 2 subgroup $\Gamma^{1}_{+}(V,q):=\Gamma^{1}(V,q)\cap\Gamma_{+}(V,q)$ in $\Gamma_{+}(V,q)$, with the non-trivial coset $\Gamma^{1}_{-}(V,q):=\Gamma^{1}(V,q)\cap\Gamma_{+}(V,q)$. \label{kerN}
\end{cor}

\begin{proof}
The proof of Lemma \ref{NinFx} shows that the norm becomes a homomorphism from $\Gamma^{\mathbb{F}^{\times}}(V,q)$, $\widetilde{\Gamma}^{\mathbb{F}^{\times}}(V,q)$, $\Gamma^{\mathbb{F}^{\times}}_{\pm}(V,q)$, or $\Gamma^{\mathbb{F}^{\times}}_{+}(V,q)$ into $\mathbb{F}^{\times}$ respectively, of which $\Gamma^{1}(V,q)$, $\widetilde{\Gamma}^{1}(V,q)$, $\Gamma^{1}_{\pm}(V,q)$ or $\Gamma^{1}_{+}(V,q)$ is the respective kernel. This proves the corollary.
\end{proof}
The group $\Gamma^{1}(V,q)$, and sometimes $\Gamma^{1}_{\pm}(V,q)$, is also called the \emph{pin group} of $(V,q)$ in the literature, while $\Gamma^{1}_{+}(V,q)$ is the \emph{Spin group} associated with that space.

\begin{rmk}
Unlike Lemma \ref{NinFx}, the restriction of $\pi$ to $\Gamma^{1}(V,q)$ or to $\Gamma^{1}_{\pm}(V,q)$, and of $\widetilde{\pi}$ to $\widetilde{\Gamma}^{1}(V,q)$ no longer surjects onto $\operatorname{O}_{V^{\perp}}(V,q)$ or $\operatorname{SO}_{V^{\perp}}(V_{\mathbb{F}},q_{\mathbb{F}})$, and the same for the map from $\Gamma^{1}_{+}(V,q)$ to $\operatorname{SO}_{V^{\perp}}(V,q)$. In fact, noting that dividing $\widetilde{Z}(C)$ or $\widetilde{Z}(C)_{+}$ by its nilpotent radical yields $\mathbb{F}$, we find that the images of norms from $\widetilde{Z}(C)^{\times}$ or from $\widetilde{Z}(C)_{+}^{\times}$ modulo nilpotents lies in $\mathbb{F}$. It follows that there are maps from $\operatorname{O}_{V^{\perp}}(V,q)$ and from $\operatorname{SO}_{V^{\perp}}(V_{\mathbb{F}},q_{\mathbb{F}})$ into $\mathbb{F}^{\times}/(\mathbb{F}^{\times})^{2}$, called \emph{spinor norm} maps, and the images of the restrictions from Corollary \ref{kerN} are the kernels $\operatorname{O}_{V^{\perp}}^{1}(V,q)$, $\operatorname{SO}_{V^{\perp}}^{1}(V_{\mathbb{F}},q_{\mathbb{F}})$, and $\operatorname{SO}_{V^{\perp}}^{1}(V,q)$ of these spinor maps (sometimes called the \emph{spinor kernel} in the first case and the \emph{special spinor kernel}, because it is already contained in an $\operatorname{SO}$-group, in the second and third ones). Recall that in the non-degenerate case we have $\widetilde{Z}(C)=\widetilde{Z}(C)_{+}=\mathbb{F}$, so that the groups from Lemma \ref{NinFx} and Corollary \ref{grClgrp} (with the index $\pm$) are the full groups $\Gamma(V,q)$ and $\widetilde{\Gamma}(M,q)$, and that with the index $+$ has index 2 there. Then the kernels of all the projections that we considered is $\mathbb{F}^{\times}$, and just $\{\pm1\}$ for the groups with superscript 1. However, in general the determination of the kernels is more complicated (they do contains $\mathbb{F}^{\times}$ and $\{\pm1\}$ respectively, and project precisely onto these subgroups modulo nilpotents), and will not be required here. \label{spinor}
\end{rmk}
The special spinor kernel form Remark \ref{spinor} is, in many cases (in particular when the dimension of $\overline{V}$ is at least 3), the commutator subgroup of $\operatorname{O}_{V^{\perp}}(V,q)$ (or its special orthogonal subgroup). This is proved in Subsection 43D of \cite{[O]} in the non-degenerate case, and since the kernel from Equation \eqref{OSES} is easily seen to be generated by commutators as well, this statement extends to degenerate quadratic spaces.

\smallskip

Let $(V,q)$ be a quadratic space over $\mathbb{F}$. We now recall three isomorphisms between Clifford algebras of extensions of $(V,q)$ (containing this space as a direct summand) and other types of algebras constructed from $\mathcal{C}(V,q)$. First, following our notation associated with paravectors, we denote by $(V_{\mathbb{F}},q_{\mathbb{F}})$ the quadratic space obtained as the orthogonal direct sum of $(V,q)$ with a 1-dimensional quadratic space of the form $\mathbb{F}\rho$ with $\rho$ having quadratic value $-1$ (i.e., elements of $V_{\mathbb{F}}$ are of the form $\xi=v+a\rho$ for $v \in V$ and $a\in\mathbb{F}$, and the $q_{\mathbb{F}}$-image of such $\xi$ is $q(v)-a^{2}$. Then, using the fact that $\rho$ squares to $-1$ and commutes with elements of $\mathcal{C}_{+}$ but anti-commutes with elements of $\mathcal{C}_{+}$ establishes the following result, mentioned in, e.g., the paragraph following Proposition 5.2 of \cite{[Mc]}.
\begin{lem}
The map taking $\alpha=\alpha_{+}+\alpha_{-}\in\mathcal{C}$, with $\alpha_{\pm}\in\mathcal{C}_{\pm}$, to $\alpha_{+}+\alpha_{-}\rho$ defines an isomorphism from $\mathcal{C}$ onto $\mathcal{C}_{\mathbb{F},+}:=\mathcal{C}(V_{\mathbb{F}},q_{\mathbb{F}})_{+}$, which we denote by $\alpha\mapsto\alpha_{\rho}$. This isomorphism commutes with the two Clifford involutions, namely we have $(\overline{\alpha})_{\rho}=\overline{\alpha_{\rho}}$ for any $\alpha\in\mathcal{C}$. On the level of invertible elements we thus get $\mathcal{C}^{\times}\cong\mathcal{C}_{\mathbb{F},+}^{\times}$. \label{isotoC+}
\end{lem}
For the last statement in Lemma \ref{isotoC+}, observe that $\overline{\alpha_{\rho}}$ equals $\overline{\alpha_{+}}-\rho\overline{\alpha_{-}}$ (since the Clifford involution interchanges the order of multiplication and takes $\rho$ to $-\rho$), and since $\rho$ anti-commutes with elements of $\mathcal{C}_{+}$, this is $\overline{\alpha_{+}}+\overline{\alpha_{-}}\rho=(\overline{\alpha})_{\rho}$.

Let $(U,h)$ be a \emph{hyperbolic plane} over $\mathbb{F}$, namely $U$ is the 2-dimensional space $\mathbb{F}e\oplus\mathbb{F}f$ with $h(e)=h(f)=0$ and $e$ and $f$ pairing to 1 (so that the value of a general element $ae+bf$ under $h$ is $ab$). We denote the quadratic space obtained as the orthogonal direct sum of $(V,q)$ and $(U,h)$ by $(V_{U},q_{U})$. For this space we cite Lemmas 3.4 and 3.5 of \cite{[Mc]}.
\begin{prop}
The map associating the matrix $\binom{\alpha\ \ \beta}{\gamma\ \ \delta}$, with $\alpha$, $\beta$, $\gamma$, and $\delta$ in $\mathcal{C}$, with $ef\alpha+e\beta'+f\gamma+fe\delta'=\alpha ef+\beta e+\gamma'f+\delta'fe$, is an isomorphism between the algebra $\operatorname{M}_{2}(\mathcal{C})$ of $2\times2$ matrices over $\mathcal{C}$ and $\mathcal{C}_{U}:=\mathcal{C}(V_{U},q_{U})$. The grading, transpose, and Clifford involutions on the latter Clifford algebra are transferred to the operations sending our matrix to $\binom{\ \ \alpha'\ \ -\beta'}{-\gamma'\ \ \ \ \delta'}$, $\binom{\overline{\delta}\ \ \overline{\beta}}{\overline{\gamma}\ \ \overline{\alpha}}$, and $\binom{\ \ \delta^{*}\ \ -\beta^{*}}{-\gamma^{*}\ \ \alpha^{*}}$ respectively. The multiplicative group $\mathcal{C}_{U}^{\times}$ is thus isomorphic to the group $\operatorname{GL}_{2}(\mathcal{C})$ of invertible $2\times2$ matrices over $\mathcal{C}$. \label{isoM2}
\end{prop}
The proof of the first assertion is a generalization of the fact that the Clifford algebra of $(U,h)$, which is spanned by $e$, $f$, $ef$, and $fe$ (with $e^{2}=f^{2}=0$ and $ef+fe=1$), is isomorphic to $\operatorname{M}_{2}(\mathbb{F})$ by identifying $ef\leftrightarrow\binom{1\ \ 0}{0\ \ 0}$, $e\leftrightarrow\binom{0\ \ 1}{0\ \ 0}$, $f\leftrightarrow\binom{0\ \ 0}{1\ \ 0}$, and $fe\leftrightarrow\binom{0\ \ 0}{0\ \ 1}$ (see Subsection 2.5 of \cite{[Ba]}), combined with the commutation relations of $e$, $f$, $fe$, and $fe$ and elements of $\mathcal{C}$. The second one is a straightforward calculation. Note that in the notation from \cite{[EGM]}, the vectors $f_{0}$ and $f_{1}$ are $f+e$ and $f-e$ respectively (the latter element is minus the \emph{Weyl element} from Section 4.1 of \cite{[Ba]}), so that $\tau_{0}=f$ and $\tau_{1}=e$, and the products $u$ and $v$ from that reference are $fe$ and $ef$ respectively.

We combine the two constructions, by setting $(V_{U,\mathbb{F}},q_{U,\mathbb{F}})$ to be the direct sum of the three spaces $(V,q)$, $(U,h)$ and $\mathbb{F}\rho$ from above. Then $f_{2}$ from \cite{[EGM]} is $\rho$, so that $w_{0}$ and $w_{1}$ there are $f\rho$ and $e\rho$ respectively, and the element $f_{0}f_{1}f_{2}$, which we denote by $\upsilon$, equals $(ef-fe)\rho=\rho(ef-fe)$. We modify the notation from Lemma \ref{isotoC+} by defining $\alpha_{\upsilon}$ to be $\alpha_{+}+\alpha_{-}\upsilon$ for $\alpha$ decomposed as above, and observe the following equalities, all holding inside the Clifford algebra $\mathcal{C}_{U,\mathbb{F}}:=\mathcal{C}(V_{U,\mathbb{F}},q_{U,\mathbb{F}})$.
\begin{lem}
We have $e\upsilon=\upsilon e=\rho e=-e\rho$, $f\upsilon=\upsilon f=-\rho f=f\rho$, and $\rho\upsilon=\upsilon\rho=fe-ef$. Thus for $\alpha\in\mathcal{C}$ we get $\alpha_{\upsilon}e=\alpha_{\rho}e$ and $\alpha_{\upsilon}f=\alpha'_{\rho}f$, and we also have $(\alpha^{*})_{\upsilon}=(\alpha_{\upsilon})^{*}$. The element $\alpha_{\rho}$ commutes with $e$ and $f$, and we have $\rho\alpha_{\rho}=\alpha'_{\rho}\rho$. On the other hand, with the index $\upsilon$ we get $e\alpha_{\upsilon}=\alpha'_{\rho}e$, $f\alpha_{\upsilon}=\alpha_{\rho}f$, and $\rho\alpha_{\upsilon}=\alpha'_{\upsilon}\rho$. Finally, if $\iota:\mathbb{F} \oplus V \to V_{\mathbb{F}}$ is the isomorphism of quadratic spaces taking $\xi=a+v$ to $v-a\rho$ then we have $\iota(\xi)=-\xi_{\rho}\rho$ for every $\xi\in\mathbb{F} \oplus V\subseteq\mathcal{C}$, and this vector $\iota(\xi)$ anti-commutes with $e$ and $f$ and satisfies $\rho\iota(\xi)=\iota(\overline{\xi})\rho$. \label{relsefrho}
\end{lem}

\begin{proof}
The first two equalities are straightforward calculations, yielding the second ones immediately. The relation with the transpositions is proved like the corresponding result in Lemma \ref{isotoC+}, via anti-commutation and the fact that $\upsilon^{*}=-\upsilon$. Then, $e$ and $f$ commute with $\mathcal{C}_{\mathbb{F},+}$, the relation with $\rho$ is as before, and the next relations follow from $e$, $f$, and $\rho$ commuting with $\mathcal{C}_{+}$ and anti-commuting with $\mathcal{C}_{+}$ and from the first relations here. As for $\iota(\xi)$, if $\xi=a+v$ then $\xi_{\rho}=a\rho+v\rho$, and multiplying by $-\rho$ from the right gives $v-a\rho=\iota(\xi)$ because $\rho^{2}=q_{\mathbb{F}}(\rho)=-1$. As $\iota(\xi)\in\mathcal{C}_{\mathbb{F},-}$ inside $\mathcal{C}_{U,\mathbb{F}}$, the anti-commutation with $e$ and $f$ follows, and the fact that $\rho\in\mathcal{C}_{\mathbb{F},-}$ implies that $\rho\iota(\xi)\rho^{-1}=-r_{\rho}\big(\iota(\xi)\big)$, where $r_{\rho}$ is the reflection in $\rho$. But for our $\xi$ the latter vector is $-a\rho-v$, which is indeed the image of $\xi'=\overline{\xi}=-r_{1}(\xi)$ under $\iota$. This proves the lemma.
\end{proof}
Combining Lemmas \ref{isotoC+} and \ref{relsefrho} with Proposition \ref{isoM2}, we obtain Proposition 2.5 of \cite{[EGM]}, which is generalized by Lemma 5.3 of \cite{[Mc]}.
\begin{prop}
The map sending the matrix $\binom{\alpha\ \ \beta}{\gamma\ \ \delta}$, with entries as in Proposition \ref{isoM2}, to $\alpha_{\rho}ef+\beta_{\rho}e\rho+\gamma'_{\rho}f\rho+\delta'_{\rho}fe=\alpha_{\upsilon}ef+\beta_{\upsilon}e\rho+\gamma_{\upsilon}f\rho+\delta_{\upsilon}fe$, yields an isomorphism from the matrix ring $\operatorname{M}_{2}(\mathcal{C})$ onto the even Clifford algebra $\mathcal{C}_{U,\mathbb{F},+}:=\mathcal{C}(V_{U,\mathbb{F}},q_{U,\mathbb{F}})_{+}$. The involution on $\mathcal{C}(V_{U,\mathbb{F}},q_{U,\mathbb{F}})_{+}$ arising from transposition of Clifford again corresponds to $\binom{\alpha\ \ \beta}{\gamma\ \ \delta}\mapsto\binom{\ \ \delta^{*}\ \ -\beta^{*}}{-\gamma^{*}\ \ \alpha^{*}}$, and we have an isomorphism between $\operatorname{GL}_{2}(\mathcal{C})$ and $\mathcal{C}_{U,\mathbb{F},+}^{\times}$. \label{isoUF}
\end{prop}
Restricting the map from Proposition \ref{isoUF} to scalar matrices $\lambda I$, representing elements $\lambda\in\mathcal{C}\subseteq\operatorname{M}_{2}(\mathcal{C})$, we obtain the map $\lambda\mapsto\lambda_{\upsilon}$ from Proposition 2.4 of \cite{[EGM]}, and the relation with transpositions from Lemma \ref{relsefrho} holds also in $\mathcal{C}_{U,\mathbb{F},+}$ (this map is also related to the one appearing in Proposition 13.23 of \cite{[P]}). This, combined with the fact that the image of the matrix under this operation involves transposition of the entries, suggests that this operation should be viewed as transposition on the matrix algebra (as in Proposition 2.5 of \cite{[EGM]}). However, considering our relations with the norm map below, and the fact that the adjoint operation on matrices corresponds to the Clifford involution (which is the main involution on quaternion algebras), our operation on matrices may still better be viewed as a Clifford involution operation.

\smallskip

We can now define the Vahlen groups. For doing so we first define the set
\begin{equation}
\mathcal{T}(V,q):=\{\alpha\in\mathcal{C}|\alpha V\alpha^{*} \subseteq V,\ N(\alpha)\in\mathbb{F}\}. \label{entry}
\end{equation}
Note that $\mathcal{T}(V,q)\cap\mathcal{C}^{\times}$ equals the set of those $\alpha\in\mathcal{T}(V,q)$ with $N(\alpha)\in\mathbb{F}^{\times}$ (or equivalently $N(\alpha)\neq0$), and as $\alpha^{*}=\alpha'^{-1}N(\alpha)$ for such $\alpha$ (since if $N(\alpha)\in\mathbb{F}$ then it equals $N(\alpha)'=N(\alpha')=\alpha'\alpha^{*}$), and $N(\alpha)$ is an invertible scalar, this intersection equals precisely the subgroup $\Gamma^{\mathbb{F}^{\times}}(V,q)$ from Lemma \ref{NinFx}.

The set $\mathcal{T}(V,q)$ from Equation \eqref{entry} is invariant under the grading involution (as applying this involution to the defining equation shows, via the fact that this involution acts as $-\operatorname{Id}$ on $V$). Then Theorems 4.1 and 3.6 of \cite{[Mc]} combine as follows.
\begin{thm}
Assume that the set $\mathcal{T}(V,q)$ from Equation \eqref{entry} is invariant under transposition, or equivalently the Clifford involution, on the Clifford algebra $\mathcal{C}:=\mathcal{C}(V,q)$ from Equation \eqref{twZClgrp}. Then, tor a matrix $\binom{\alpha\ \ \beta}{\gamma\ \ \delta}$ in the matrix ring $\operatorname{M}_{2}(\mathcal{C})$, the following conditions are equivalent.
\begin{enumerate}
\item $\alpha$, $\beta$, $\gamma$, and $\delta$ have norms in $\mathbb{F}$; $\alpha\beta^{*}=\beta\alpha^{*}$; $\gamma\delta^{*}=\delta\gamma^{*}$; $\alpha\delta^{*}-\beta\gamma^{*}\in\mathbb{F}^{\times}$; $\alpha\overline{\gamma}$ and $\beta\overline{\delta}$ are in $V$; And for every $v \in V$, the elements $\alpha v\overline{\beta}+\beta\overline{v}\overline{\alpha}$ and $\gamma v\overline{\delta}+\delta\overline{v}\overline{\gamma}$ are in $\mathbb{F}$ and $\alpha v\overline{\delta}+\beta\overline{v}\overline{\gamma} \in V$.
\item $\alpha$, $\beta$, $\gamma$, and $\delta$ are in $\mathcal{T}(V,q)$; $\alpha\delta^{*}-\beta\gamma^{*}\in\mathbb{F}^{\times}$; And $\alpha\beta^{*}$ and $\delta\gamma^{*}$ are in $V$.
\item $\alpha$, $\beta$, $\gamma$, and $\delta$ are in $\mathcal{T}(V,q)$; $\alpha\delta^{*}-\beta\gamma^{*}\in\mathbb{F}^{\times}$; And $\overline{\alpha}\beta$ and $\overline{\delta}\gamma$ are in $V$.
\item The image of this matrix under the isomorphism from Proposition \ref{isoM2} lies in the group $\Gamma^{\mathbb{F}^{\times}}(V_{U},q_{U})$ from Lemma \ref{NinFx} for the quadratic space $(V_{U},q_{U})$.
\end{enumerate}
When these conditions are satisfied, the value $\alpha\delta^{*}-\beta\gamma^{*}\in\mathbb{F}^{\times}$ equals the norm of the corresponding element of $\mathcal{C}_{U}$. \label{Vahlen}
\end{thm}
It follows from Condition 4 of Theorem \ref{Vahlen} that the set of matrices satisfying any of the other equivalent conditions there form a multiplicative group (i.e., a subgroup of $\operatorname{GL}_{2}(\mathcal{C})$). This group is called the \emph{Vahlen group} associated with $(V,q)$, and is denoted by $\operatorname{V}(V,q)$. The equivalence of Conditions 1 and 4 in that theorem is independent of the assumption on $\mathcal{T}(V,q)$, and the theorem yields an isomorphism of groups between $\operatorname{V}(V,q)$ and $\Gamma^{\mathbb{F}^{\times}}(V_{U},q_{U})$.

The last assertion in Theorem \ref{Vahlen} produces the following consequence.
\begin{cor}
The map taking a matrix $\binom{\alpha\ \ \beta}{\gamma\ \ \delta}\in\operatorname{V}(V,q)$ to $\alpha\delta^{*}-\beta\gamma^{*}\in\mathbb{F}^{\times}$ is a group homomorphism. Its kernel consists of matrices satisfying a set of equivalent conditions similar to those from Theorem \ref{Vahlen}, but with each instance of $\alpha\delta^{*}-\beta\gamma^{*}\in\mathbb{F}^{\times}$ replaces by $\alpha\delta^{*}-\beta\gamma^{*}=1$, and with the subgroup of $\mathcal{C}_{U}^{\times}$ being the (pin) group $\Gamma^{1}(V_{U},q_{U})$ from Corollary \ref{kerN}. \label{detSV}
\end{cor}
The map from Corollary \ref{detSV} is called the \emph{pseudo-determinant} and denoted by $\det$, and its kernel is the \emph{special Vahlen group} $\operatorname{SV}(V,q)$.

\smallskip

The Vahlen group construction has a paravector analogue. Similarly to Equation \eqref{entry}, we define
\begin{equation}
\widetilde{\mathcal{T}}(V,q):=\{\alpha\in\mathcal{C}|\alpha(\mathbb{F} \oplus V)\alpha^{*}\subseteq\mathbb{F} \oplus V,\ N(\alpha)\in\mathbb{F}\}. \label{enpara}
\end{equation}
Also in $\widetilde{\mathcal{T}}(V,q)$, the invertible elements, i.e., those having norms in $\mathbb{F}^{\times}$ (or equivalently non-zero norms) are precisely the elements of the subgroup $\widetilde{\Gamma}^{\mathbb{F}^{\times}}(V,q)$ from Lemma \ref{NinFx}.

As the grading involution is an automorphism of $\mathbb{F} \oplus V$ (namely $-r_{1}$), it leaves the set $\widetilde{\mathcal{T}}(V,q)$ from Equation \eqref{enpara} invariant as well. The combination of Theorems 5.7 and 6.1 of \cite{[Mc]}, the latter generalizing Theorem 3.7 of \cite{[EGM]} from the case called strongly anisotropic in that reference, gives the following result.
\begin{thm}
Let a matrix $\binom{\alpha\ \ \beta}{\gamma\ \ \delta}\in\operatorname{M}_{2}(\mathcal{C})$ be given, for $\mathcal{C}:=\mathcal{C}(V,q)$, and assume that transposition, or equivalently the Clifford involution, preserves the set $\widetilde{\mathcal{T}}(V,q)$ from Equation \eqref{enpara}. Then the following four conditions are equivalent.
\begin{enumerate}
\item $\alpha$, $\beta$, $\gamma$, and $\delta$ have norms in $\mathbb{F}$; $\alpha\beta^{*}=\beta\alpha^{*}$; $\gamma\delta^{*}=\delta\gamma^{*}$; $\alpha\delta^{*}-\beta\gamma^{*}\in\mathbb{F}^{\times}$; $\alpha\overline{\gamma}$ and $\beta\overline{\delta}$ are in $\mathbb{F} \oplus V$; And for every $\xi\in\mathbb{F} \oplus V$, the elements $\alpha\xi\overline{\beta}+\beta\overline{\xi}\overline{\alpha}$ and $\gamma\xi\overline{\delta}+\delta\overline{\xi}\overline{\gamma}$ are in $\mathbb{F}$ and $\alpha\xi\overline{\delta}+\beta\overline{\xi}\overline{\gamma}\in\mathbb{F} \oplus V$.
\item $\alpha$, $\beta$, $\gamma$, and $\delta$ are in $\widetilde{\mathcal{T}}(V,q)$; $\alpha\delta^{*}-\beta\gamma^{*}\in\mathbb{F}^{\times}$; And $\alpha\beta^{*}$ and $\delta\gamma^{*}$ are in $\mathbb{F} \oplus V$.
\item $\alpha$, $\beta$, $\gamma$, and $\delta$ are in $\widetilde{\mathcal{T}}(V,q)$; $\alpha\delta^{*}-\beta\gamma^{*}\in\mathbb{F}^{\times}$; And $\overline{\alpha}\beta$ and $\overline{\delta}\gamma$ are in $\mathbb{F} \oplus V$.
\item The image of this matrix under the isomorphism from Proposition \ref{isoUF} lies in the group $\Gamma^{\mathbb{F}^{\times}}_{+}(V_{U,\mathbb{F}},q_{U,\mathbb{F}})$ associated by Corollary \ref{Normgrade} with the quadratic space $(V_{U,\mathbb{F}},q_{U,\mathbb{F}})$.
\end{enumerate}
Also in this case $\alpha\delta^{*}-\beta\gamma^{*}\in\mathbb{F}^{\times}$ is the norm of the associated element of $\mathcal{C}_{U,\mathbb{F},+}$. \label{pvVah}
\end{thm}
Once again, the set of matrices satisfying each one of the equivalent conditions from Theorem \ref{pvVah} is a (multiplicative) subgroup of $\operatorname{GL}_{2}(\mathcal{C})$, by Condition 4 there. This is the \emph{paravector Vahlen group} $\widetilde{\operatorname{V}}(V,q)$ that is associated with $(V,q)$. The additional invariance of $\widetilde{\mathcal{T}}(V,q)$ is not required for the equivalence of Conditions 1 and 4 in Theorem \ref{pvVah} as well, and the latter condition shows that $\widetilde{\operatorname{V}}(V,q)$ is isomorphic, as a group, to $\Gamma^{\mathbb{F}^{\times}}_{+}(V_{U,\mathbb{F}},q_{U,\mathbb{F}})$.

We also have the following analogue of Corollary \ref{detSV}.
\begin{cor}
We have a pseudo-determinant map $\det:\widetilde{\operatorname{V}}(V,q)\to\mathbb{F}^{\times}$ taking the matrix from Theorem \ref{pvVah} again to $\alpha\delta^{*}-\beta\gamma^{*}$, which is a group homomorphism. The kernel is characterized by conditions like in that theorem, with $\alpha\delta^{*}-\beta\gamma^{*}=1$ instead of replacing $\alpha\delta^{*}-\beta\gamma^{*}\in\mathbb{F}^{\times}$, and with the (spin) group $\Gamma^{1}_{+}(V_{U,\mathbb{F}},q_{U,\mathbb{F}})$ inside $\mathcal{C}_{U,\mathbb{F},+}^{\times}$. \label{SVpara}
\end{cor}
The kernel from Corollary \ref{SVpara} is called the \emph{special paravector Vahlen group} $\widetilde{\operatorname{SV}}(V,q)$. This corollary is the generalization of Theorem 4.1 of \cite{[EGM]}.

We conclude with some extra properties of the entries of an element of the Vahlen and paravector Vahlen groups.
\begin{rmk}
If $\binom{\alpha\ \ \beta}{\gamma\ \ \delta}$ is a matrix in $\operatorname{M}_{2}(\mathcal{C})$ that lies in either $\operatorname{V}(V,q)$ or $\widetilde{\operatorname{V}}(V,q)$, then $\overline{\alpha}$, $\overline{\beta}$, $\overline{\gamma}$, and $\overline{\delta}$ have the same norms as $\alpha$, $\beta$, $\gamma$, and $\delta$ respectively. This follows from the formula for the inverse of a our matrix being $\frac{1}{\alpha\delta^{*}-\beta\gamma^{*}}\binom{\alpha\ \ \beta}{\gamma\ \ \delta}\mapsto\binom{\ \ \delta^{*}\ \ -\beta^{*}}{-\gamma^{*}\ \ \alpha^{*}}$. Moreover, if our matrix is in $\operatorname{V}(V,q)$ (resp. $\widetilde{\operatorname{V}}(V,q)$) and $v$ is in $V$ (resp. $\xi$ is in $\mathbb{F} \oplus V$) then the scalars $\alpha v\overline{\beta}+\beta\overline{v}\overline{\alpha}$ and $\gamma v\overline{\delta}+\delta\overline{v}\overline{\gamma}$ (resp. $\alpha\xi\overline{\beta}+\beta\overline{\xi}\overline{\alpha}$ and $\gamma\xi\overline{\delta}+\delta\overline{\xi}\overline{\gamma}$) from Condition 1 of Theorem \ref{Vahlen} (resp. Theorem \ref{pvVah}) are also equal to minus the pairing of $v$ (resp. $\xi$) with $\overline{\alpha}\beta$ and $\overline{\gamma}\delta$. To see this, recall that as elements of $V$ or of $\mathbb{F} \oplus V$, the elements $\overline{\beta}\alpha$ and $\overline{\delta}\gamma$ equal $\overline{\overline{\alpha}\beta}=(\overline{\alpha}\beta)'$ and $\overline{\overline{\gamma}\delta}=(\overline{\gamma}\delta)'$ respectively. We multiply the first scalar from Theorem \ref{Vahlen} or \ref{pvVah} by $\alpha$ and by $\beta$ from the right and multiply the formula for the first pairing in Equation \eqref{pairC} or \eqref{parabil} by these elements of $\mathcal{T}(V,q)$ or $\widetilde{\mathcal{T}}(V,q)$ from the left, do the same for the second scalar and pairing with $\gamma$ and $\delta$, observe that we obtain the same values in all these operations, and use the invertibility of $\alpha\delta^{*}-\beta\gamma^{*}=\delta\alpha^{*}-\gamma\beta^{*}$ (this equality follows from the fact that this element is in $\mathbb{F}$ and is thus invariant under transposition) for obtaining the required equality. \label{invV}
\end{rmk}

\section{Group Actions via M\"{o}bius Transformations \label{ActVah}}

The construction of the hyperbolic half-space, as a symmetric space of the orthogonal group of a Lorentzian space of signature $(n,1)$, is based on presentations of vectors of norm $-1$ using related spaces. This is the relation between the hyperboloid and half-space models of this symmetric space. The action on the latter space is given in terms of M\"{o}bius transformations. We now generalize this construction to our more general setting, with emphasis on the precise form of the boundary elements that one must add for obtaining a symmetric space with a well-defined action. In this section we do it for the Vahlen group as defined in Theorem \ref{Vahlen}, and the paravector Vahlen group from Theorem \ref{pvVah} will be considered in the next section.

For this we shall make the action of the Vahlen group $\operatorname{V}(V,q)$ defined via Theorem \ref{Vahlen} on the quadratic space $(V_{U},q_{U})$ appearing in Proposition \ref{isoM2} more explicit. Recall that the Clifford algebra $\mathcal{C}_{U}$ that is associated with $(V_{U},q_{U})$ is acted by its group of units $\mathcal{C}_{U}^{\times}$ by $\eta\in\mathcal{C}_{U}^{\times}$ taking $\psi\in\mathcal{C}_{U}$ to $\eta\psi\eta^{*}$. We shall restrict attention to $\psi \in V_{U}$, where the latter space is decomposed as the direct sum of $V$, $\mathbb{F}e$, and $\mathbb{F}f$, and present $\eta$ via the isomorphism from Proposition \ref{isoM2}. As usual, $\mathcal{C}$ stands for the Clifford algebra $\mathcal{C}(V,q)$ of our original space $(V,q)$.
\begin{lem}
If $\eta$ corresponds to the matrix $\binom{\alpha\ \ \beta}{\gamma\ \ \delta}$, where $\alpha$, $\beta$, $\gamma$, and $\delta$ are in $\mathcal{C}$, then the action of $\eta$ takes $e$ to $N(\alpha)e+N(\gamma)'f+\alpha\overline{\gamma}ef+(\gamma\overline{\alpha})'fe$, and $f$ to $N(\beta)e+N(\delta)'f+\beta\overline{\delta}ef+(\delta\overline{\beta})'fe$. Moreover, any $u \in V$ is sent by this action to \[(\alpha u\overline{\beta}+\beta\overline{u}\overline{\alpha})e+(\gamma u\overline{\delta}+\delta\overline{u}\overline{\gamma})'f+(\alpha u\overline{\delta}+\beta\overline{u}\overline{\gamma})ef+(\gamma u\overline{\beta}+\delta\overline{u}\overline{\alpha})'fe.\] \label{actVU}
\end{lem}

\begin{proof}
The element of $\mathcal{C}_{U}^{\times}$ that is associated with $\eta$ is $\alpha ef+\beta e+\gamma'f+\delta'fe$ by the formula from Proposition \ref{isoM2}, and $\eta^{*}$ is thus $fe\alpha^{*}+e\beta^{*}+f\overline{\gamma}+ef\overline{\delta}$. Recalling that $e^{2}=q(e)=0$ in $\mathcal{C}_{U}$, so that $efe=e(ef+fe)=e$ (since $ef+fe=(e,f)=1$), the product $\eta e\eta^{*}$ reduces to $\alpha e\alpha^{*}+\alpha ef\overline{\gamma}+\gamma'fe\alpha^{*}+\gamma'fef\overline{\gamma}$. But with $f^{2}=q(f)=0$ we get $fef=f$ as well, and $e$ and $f$ satisfy a commutation relation like in Equation \eqref{twZClgrp} with elements of $\mathcal{C}$, so that the definition of the norm and its evident commutation with the grading involution yields the desired expression. Similarly, the expression for $\eta f\eta^{*}$ is $\beta efe\beta^{*}+\beta ef\overline{\delta}+\delta'fe\beta^{*}+\delta'f\overline{\delta}$, which becomes the required expression by similar considerations.

Consider now $u \in V$, and recall that $eu=-ue$ and $fu=-uf$ in $\mathcal{C}_{U}$, so that $ef$ and $fe$ commute with $u$. This presents $\eta u\eta^{*}$ as \[\alpha uef\overline{\delta}+\alpha ue\beta^{*}-\beta ue\alpha^{*}-\beta uef\overline{\gamma}-\gamma'ufe\beta^{*}-\gamma'uf\overline{\delta}+\delta'ufe\alpha^{*}+\delta'uf\overline{\gamma}.\] The commutation relations between $e$, $f$, and elements of $\mathcal{C}$, and the fact that transposition is the composition of the grading and Clifford involution and we have $u'=\overline{u}=-u$ for $u \in V$ transform the latter expression into the asserted formula. This proves the lemma.
\end{proof}
In fact, we do not require $\eta$ to be invertible in Lemma \ref{actVU}, and the same formula holds equally well for every $\eta\in\mathcal{C}_{U}$. We are, however, more interested in the action of the Vahlen group $\operatorname{V}(V,q)$.
\begin{cor}
If the matrix $\binom{\alpha\ \ \beta}{\gamma\ \ \delta}$ lies in $\operatorname{V}(V,q)$, then its action via Lemma \ref{actVU} takes $e$, $f$, and $u \in V$ to $\alpha\overline{\gamma}+N(\alpha)e+N(\gamma)f$, $\beta\overline{\delta}+N(\beta)e+N(\delta)f$, and $(\alpha u\overline{\delta}+\beta\overline{u}\overline{\gamma})+(\alpha u\overline{\beta}+\beta\overline{u}\overline{\alpha})e+(\gamma u\overline{\delta}+\delta\overline{u}\overline{\gamma})f$ respectively, all of which lie in $V_{U}\subseteq\mathcal{C}_{U}$. \label{VUact}
\end{cor}

\begin{proof}
Condition 1 of Theorem \ref{Vahlen} implies that $N(\alpha)$, $N(\beta)$, $\alpha u\overline{\beta}+\beta\overline{u}\overline{\alpha}$, $N(\gamma)$, $N(\delta)$, and $\gamma u\overline{\delta}+\delta\overline{u}\overline{\gamma}$ appearing in the coefficients of $e$ in $f$ in Lemma \ref{actVU}, are all elements of $\mathbb{F}$ (so that in particular the grading involution leaves the latter three elements invariant). Moreover, $\alpha\overline{\gamma}$, $\beta\overline{\delta}$, and $\alpha u\overline{\delta}+\beta\overline{u}\overline{\gamma}$, appearing as coefficients of $ef$, all lie in $V$, so that they are all invariant under applying the Clifford involution and then the grading involution, which yields the respective coefficients of $fe$. As these vectors are multiplied by $ef+fe=1$, we obtain the asserted expressions, which therefore indeed lie in $V\oplus\mathbb{F}e\oplus\mathbb{F}f=V_{U}$ as desired. This proves the corollary.
\end{proof}
The part of Corollary \ref{VUact} stating that this action of $\operatorname{V}(V,q)$ preserves $V_{U}$ also follows from the isomorphism from Theorem \ref{Vahlen}, the relation with the norm there (which implies that $\eta^{*}$ is the scalar $\det\binom{\alpha\ \ \beta}{\gamma\ \ \delta}$ from Corollary \ref{detSV} times $\eta'^{-1}$), and the definition of $\Gamma^{\mathbb{F}^{\times}}(V_{U},q_{U})$ in Lemma \ref{NinFx} and Equation \eqref{twZClgrp}. However, we shall need the explicit formulae appearing in that corollary as well.

\smallskip

As the space $V_{U}^{\perp}$ defined as in Equation \eqref{Vperp} is the image of $V^{\perp}$ in $V_{U}$, just like for $V_{\mathbb{F}}$ above, the group associated with $(V_{U},q_{U})$ in Equation \eqref{orthgrp} is $\operatorname{O}_{V^{\perp}}(V_{U},q_{U})$, and was seen in Lemma \ref{NinFx} to be the image of $\Gamma^{\mathbb{F}^{\times}}(V_{U},q_{U})$, and therefore also of its isomorph $\operatorname{V}(V,q)$ from Theorem \ref{Vahlen}. Since it preserves the value of $q_{U}$, we denote, for every $c\in\mathbb{F}$, set of all vectors $w \in V_{U} \setminus V^{\perp}$ with $q_{U}(w)=c$ by $\mathbf{K}_{V,q}^{c}$ (the restriction that $w \not\in V^{\perp}$ affects only the definition of $\mathbf{K}_{V,q}^{0}$, of course). We shall also need the direct sum $V_{\mathbb{F}}^{c}$ of $V$ with the 1-dimensional space $\mathbb{F}\sigma_{c}$, with the quadratic form $q_{\mathbb{F}}^{c}$ making these subspaces orthogonal, restricting to $q$ on $V$, and takes $\sigma_{c}$ to $-c$ (so that if $z=v+t\sigma_{c}$ is a general element of $V_{\mathbb{F}}^{c}$, with $v \in V$ and $t\in\mathbb{F}$, then $q_{\mathbb{F}}^{c}(z)=q(v)-ct^{2}$). In particular, the space $V_{\mathbb{F}}$ considered above is $V_{\mathbb{F}}^{1}$, with $\rho=\sigma_{1}$, i.e., the special case with $c=1$.

The following lemma relates (Zariski open) subsets of $\mathbf{K}_{V,q}^{c}$ and $V_{\mathbb{F}}^{c}$.
\begin{lem}
Let $\mathbf{K}_{V,q}^{c,o}$ denote the set of those $w\in\mathbf{K}_{V,q}^{c}$ whose pairing with $e$ in $V_{U}$ is non-zero. Then the elements of $\mathbf{K}_{V,q}^{c,o}$ are precisely the vectors of the form $w=\frac{v+f+(ct^{2}-q(v))e}{t}$ for $v \in V$ and $t\in\mathbb{F}^{\times}$. They are in one-to-one correspondence with elements of the set-theoretic complement $\mathbf{H}_{V,q}^{c,o}:=V_{\mathbb{F}}^{c} \setminus V$ of $V$ inside $V_{\mathbb{F}}^{c}$, in which our $w\in\mathbf{K}_{V,q}^{c,o}$ is associated with $z=v+t\sigma_{c}\in\mathbf{H}_{V,q}^{c,o}$, again with $v \in V$ and $t\in\mathbb{F}^{\times}$. \label{normc}
\end{lem}

\begin{proof}
The non-zero pairing of $w$ with $e$ can be written as $\frac{1}{t}$ for a unique $t\in\mathbb{F}^{\times}$. Then we can write our vector $w$ as $\frac{v+f+se}{t}$ for $v \in V$ and $s\in\mathbb{F}$, and since the $q_{U}$-image of this $w$ is $\frac{q(v)+s}{t^{2}}$, comparing it with $c$ implies that $s=ct^{2}-q(v)$ as asserted. The unique presentation of $w$, as well as $z$, in terms of $v \in V$ and $t\in\mathbb{F}^{\times}$ shows that the correspondence between $\mathbf{K}_{V,q}^{c,o}$ and $\mathbf{H}_{V,q}^{c,o}$ is bijective. This proves the lemma.
\end{proof}

The M\"{o}bius transformation formula makes use of the following result.
\begin{lem}
For $z \in V_{\mathbb{F}}^{c}$ and $\eta=\binom{\alpha\ \ \beta}{\gamma\ \ \delta}\in\operatorname{V}(V,q)$, the norms $N(\gamma z+\delta)$ and $N(\alpha z+\beta)$ lie in $\mathbb{F}$, and $(\alpha z+\beta)\overline{(\gamma z+\delta)}$ lies in $V_{\mathbb{F}}^{c}$. Moreover, if $z$ lies in $\mathbf{H}_{V,q}^{c,o}$ then so does the latter vector. \label{Ncz+d}
\end{lem}

\begin{proof}
We recall that $z=v+t\sigma_{c}$, and that in $\mathcal{C}(V_{\mathbb{F}}^{c},q_{\mathbb{F}}^{c})$, the vector $\sigma_{c}$ is inverted by the Clifford involution and satisfies $\sigma_{c}\overline{\gamma}=\gamma^{*}\sigma_{c}$ for the image of $\gamma\in\mathcal{C}$ inside $\mathcal{C}(V_{\mathbb{F}}^{c},q_{\mathbb{F}}^{c})$ (and similarly for $\alpha$, $\beta$, and $\delta$). Then $N(\gamma z+\delta)=(\gamma z+\delta)\overline{(\gamma z+\delta)}$ equals \[\gamma N(z)\overline{\gamma}+\gamma v\overline{\delta}+\delta\overline{v}\overline{\gamma}+t(\gamma\delta^{*}-\delta\gamma^{*})\sigma_{c}+\delta\overline{\delta}=-q_{\mathbb{F}}^{c}(z)N(\gamma)+(\gamma v\overline{\delta}+\delta\overline{v}\overline{\gamma})+N(\delta)\] (because $\gamma\delta^{*}=\delta\gamma^{*}$ in Condition 1 of Theorem \ref{Vahlen} and $N(z)=-q_{\mathbb{F}}^{c}(z)$ for $z \in V_{\mathbb{F}}^{c}$, which lies in $\mathbb{F}$ and is hence central), and this expression lies in $\mathbb{F}$ by Condition 1 of Theorem \ref{Vahlen}. The same argument proves that \[N(\alpha z+\beta)=-q_{\mathbb{F}}^{c}(z)N(\alpha)+(\alpha v\overline{\beta}+\beta\overline{v}\overline{\alpha})+N(\beta)\in\mathbb{F}\] as well, and similar considerations show that \[(\alpha z+\beta)\overline{(\gamma z+\delta)}=-q_{\mathbb{F}}^{c}(z)\alpha\overline{\gamma}+(\alpha v\overline{\delta}+\beta\overline{v}\overline{\gamma})+t(\alpha\delta^{*}-\beta\gamma^{*})\sigma_{c}+\beta\overline{\delta},\] where the first, second, and fourth terms lie in $V$ and the third one is a non-zero scalar multiple of $t\sigma_{c}$ by Theorem \ref{Vahlen}. This proves the lemma.
\end{proof}

\begin{cor}
Consider the matrix $\eta$ and the element $z=v+t\sigma_{c}$ from Lemma \ref{normc}. Then the $q$-value of the $V$-part of the vector $(\alpha z+\beta)\overline{(\gamma z+\delta)} \in V_{\mathbb{F}}^{c}$ from that lemma is $ct^{2}\det^{2}\eta-N(\alpha z+\beta)N(\gamma z+\delta)$, with $\det\eta$ being the element $\alpha\delta^{2}-\beta\gamma^{*}\in\mathbb{F}^{\times}$ from Theorem \ref{Vahlen} and Corollary \ref{detSV}. \label{MoebVq}
\end{cor}

\begin{proof}
If is clear that $q_{\mathbb{F}}^{c}\big((\alpha z+\beta)\overline{(\gamma z+\delta)}\big)$ equals minus the norm of that vector, which is thus $-N(\alpha z+\beta)N(\gamma z+\delta)$ by the multiplicativity of the norm when it is a scalar (the fact that the latter multiplier equals the norm of $\overline{\gamma z+\delta}$ as well is established via Remark \ref{invV}). But we saw in the proof of Lemma \ref{normc} that this vector equals its $V$-part plus $t\det\eta\cdot\sigma_{c}$, so that its $q$-image, which is the same as its $q_{\mathbb{F}}^{c}$-image, is that of $(\alpha z+\beta)\overline{(\gamma z+\delta)}$ minus that of $t\det\eta\cdot\sigma_{c}$. The asserted value is thus obtained from what we just proved and the value $-c$ of $q_{\mathbb{F}}^{c}(\sigma_{c})$. This proves the corollary.
\end{proof}

We can now establish the regular part of the M\"{o}bius transformation formula.
\begin{prop}
Let $w\in\mathbf{K}_{V,q}^{c,o}$ be attached to $z\in\mathbf{H}_{V,q}^{c,o}$ via Lemma \ref{normc}, and consider a matrix $\binom{\alpha\ \ \beta}{\gamma\ \ \delta}\in\operatorname{V}(V,q)$ such that the norm $N(\gamma z+\delta)$ from Lemma \ref{Ncz+d} does not vanish, and set $\eta\in\Gamma^{\mathbb{F}^{\times}}(V_{U},q_{U})$ to be the element associated with this matrix via Theorem \ref{Vahlen}. Then the image of $w$ under the orthogonal map $\pi_{U}(\eta)\in\operatorname{O}_{V^{\perp}}(V_{U},q_{U})$, with $\pi_{U}$ being the map from Theorem \ref{GammaO} that is associated with the quadratic space $(V_{U},q_{U})$, is the element of $\mathbf{K}_{V,q}^{c,o}$ that is associated with $(\alpha z+\beta)(\gamma z+\delta)^{-1}\in\mathbf{H}_{V,q}^{c,o}$. \label{regMoeb}
\end{prop}

\begin{proof}
The map $\pi_{U}(\eta)$ takes $w$ to $\eta w\eta'^{-1}$, and since it is clear from Lemma \ref{NinFx} that $\eta^{-1}=\overline{\eta}/\det\eta$ where $\det\eta$ is the scalar $\alpha\delta^{*}-\beta\gamma^{*}$ from Theorem \ref{Vahlen} and Corollary \ref{detSV}, and this scalar is invariant under the grading involution, the expression that we seek is $\frac{\eta w\eta^{*}}{\det\eta}$. The latter expression is $\eta\big[v+f+\big(ct^{2}-q(v)\big)e\big]\eta^{*}$ divided by $t\det\eta=t(\alpha\delta^{*}-\beta\gamma^{*})$, and using Corollary \ref{VUact} we evaluate the numerator as \[\big[(\alpha v\overline{\delta}+\beta\overline{v}\overline{\gamma})+\beta\overline{\delta}+\big(ct^{2}-q(v)\big)\alpha\overline{\gamma}\big]+\big[(\alpha v\overline{\beta}+\beta\overline{v}\overline{\alpha})+N(\beta)+\big(ct^{2}-q(v)\big)N(\alpha)\big]e+\] \[+\big[(\gamma v\overline{\delta}+\delta\overline{v}\overline{\gamma})+N(\delta)+\big(ct^{2}-q(v)\big)N(\gamma)\big]f.\] Recalling that the element $z\in\mathbf{H}_{V,q}^{c,o}$ that is associated with $w$ via Lemma \ref{normc} is $v+t\sigma_{c}$, with $q_{\mathbb{F}}^{c}(z)=q(v)-ct^{2}$, the multiplier of $N(\gamma)$ in the coefficient of $f$ in the latter expression is $-q_{\mathbb{F}}^{c}(z)$, so that the full coefficient of $f$ here is the norm $N(\gamma z+\delta)$ from Lemma \ref{Ncz+d}. Moreover, the same lemma presents the $V$-part of our expression as $(\alpha z+\beta)\overline{(\gamma z+\delta)}-t(\alpha\delta^{*}-\beta\gamma^{*})\sigma_{c}$, or more precisely the $V$-part of $(\alpha z+\beta)\overline{(\gamma z+\delta)}$.

Now, the action of $\pi_{U}(\eta)$ takes $w\in\mathbf{K}_{V,q}^{c,o}$ to an element of $\mathbf{K}_{V,q}^{c}$, and as the norm from Lemma \ref{Ncz+d} does not vanish in our assumption, we deduce that $\pi_{U}(\eta)(w)\in\mathbf{K}_{V,q}^{c,o}$. Moreover, the coefficient of $f$ in it is the inverse of $\frac{t(\alpha\delta^{*}-\beta\gamma^{*})}{N(\gamma z+\delta)}$, and multiplying the $V$-part of it by the inverse of the coefficient of $f$ (for obtaining a presentation like in Lemma \ref{normc}) yields $\frac{(\alpha z+\beta)\overline{(\gamma z+\delta)}-t(\alpha\delta^{*}-\beta\gamma^{*})\sigma_{c}}{N(\gamma z+\delta)}$. On the other hand, for obtaining the element $(\alpha z+\beta)(\gamma z+\delta)^{-1}$ of $\mathbf{H}_{V,q}^{c,o}$, we divide the element $(\alpha z+\beta)\overline{(\gamma z+\delta)}$ from Lemma \ref{Ncz+d}, by $N(\gamma z+\delta)$, and obtain a vector whose $V$-part is the latter quotient, and in which the coefficient of $\sigma_{c}$ is the former quotient. Thus $\pi_{U}(\eta)(w)\in\mathbf{K}_{V,q}^{c,o}$ indeed corresponds, via Lemma \ref{normc}, to $(\alpha z+\beta)(\gamma z+\delta)^{-1}\in\mathbf{H}_{V,q}^{c,o}$ as desired. This proves the proposition.
\end{proof}

\begin{rmk}
The coefficient of $e$ in the formula for $\eta\big[v+f+\big(ct^{2}-q(v)\big)e\big]\eta^{*}$ can be seen, by an argument analogous to the proof of Lemma \ref{Ncz+d}, to be the norm $N(\alpha z+\beta)$. Thus the fact that the $q_{U}$-value of the latter vector is $ct^{2}\det^{2}\eta$ follows directly from Corollary \ref{MoebVq}, without invoking the orthogonality of $\pi_{U}(\eta)$. The latter can also be shown directly using the conditions from Theorem \ref{Vahlen} and arguments like in Remark \ref{invV}. \label{expcal}
\end{rmk}

\smallskip

If $V$ contains no vectors with $q$-value $c$, then every element of $\mathbf{K}_{V,q}^{c}$ pairs non-trivially with $e$ (for otherwise it is of the form $v+ae$ for some $v \in V$ and $a\in\mathbb{F}$, with $q_{U}$-value $q(v) \neq c$), so that $\mathbf{K}_{V,q}^{c}=\mathbf{K}_{V,q}^{c,o}$ is identified with $\mathbf{H}_{V,q}^{c,o}$ as in Lemma \ref{normc}, and the action via (regular) M\"{o}bius transformations from Proposition \ref{regMoeb} describes the full action of $\operatorname{O}_{V^{\perp}}(V_{U},q_{U})$ (or $\Gamma^{\mathbb{F}^{\times}}(V_{U},q_{U})$, or $\operatorname{V}(V,q)$). However, when $(V,q)$ does represent $c$ (which is always the case if $c=0$, since we allow $v=0$ for representations here), there exist elements of $\mathbf{K}_{V,q}^{c}\setminus\mathbf{K}_{V,q}^{c,o}$, and we wish to extend $\mathbf{H}_{V,q}^{c,o}$ to a space $\mathbf{H}_{V,q}^{c}$ on which the action of the M\"{o}bius transformations becomes fully well-defined. Section 19 of \cite{[L]} mentions such extensions, but in that reference the usual action of M\"{o}bius transformation, with one single point at $\infty$, is considered, and it does not seem to extend in the desirable way. A few cases (over $\mathbb{R}$) are worked out in \cite{[Ma]}, but from a very different viewpoint.

In order to do so, we extend $V_{\mathbb{F}}^{c}$ by adding certain vectors at $\infty$. In these vectors the multiplier of $\sigma_{c}$ is $\infty$, and their $V$-part is an infinite multiple of a vector $u \in V$ with $q(u)=c$, with a value determined, in some sense, by an ``infinitesimal translation'' of $u$, but only the value of the pairing of this translation with $u$ matters. The more precise details are given in the following definition.
\begin{defn}
The \emph{boundary point} $(\infty u)_{b}+\infty\sigma_{c}$ of $V_{\mathbb{F}}^{c}$ that associated with $u \in V$ with $q(u)=c$ and $b\in\mathbb{F}$ is defined such that the value that the extension of $q$ takes on $\infty u$, which should be $\infty^{2}c$, is defined to be $\infty^{2}c-\infty b$. When $c=0$ and the values of this extension of $q$ are only ``linear in $\infty$'' (rather than quadratic), the vectors  $u \in V^{\perp}$ (and in particular $u=0$) are included, but then we allow its extended $q$-value to only be $-\infty b$ for $b\neq0$. The set of boundary points will be denoted by $\partial\overline{V}_{\mathbb{F}}^{c}$, and the union of $V_{\mathbb{F}}^{c}$ and $\partial\overline{V}_{\mathbb{F}}^{c}$, called the \emph{completion} of the former, will be denoted by $\overline{V}_{\mathbb{F}}^{c}$. We also set $\mathbf{H}_{V,q}^{c}:=\overline{V}_{\mathbb{F}}^{c} \setminus V=\mathbf{H}_{V,q}^{c,o}\cup\partial\overline{V}_{\mathbb{F}}^{c}$. \label{boundary}
\end{defn}
Note that in case $V_{\mathbb{F}}^{c}$ comes with a topology (e.g., Zariski, or the natural one when $\mathbb{F}$ is a topological field), then we extend that topology to the space $\overline{V}_{\mathbb{F}}^{c}$ from Definition \ref{boundary} by defining a basis for the topology at a point $(\infty u)_{b}+\infty\sigma_{c}$ as follows: The intersection of a basic open set with $V_{\mathbb{F}}^{c}$ consists of those $v+t\sigma_{c}$ with $t$ in a neighborhood of $\infty$ in $\mathbb{P}^{1}(\mathbb{F})$ and $(tu-v,u)$ in a neighborhood of $b$ in $\mathbb{F}$, and its intersection with $\partial\overline{V}_{\mathbb{F}}^{c}$ consists of those $(\infty w)_{a}+\infty\sigma_{c}$ with $w$ in a neighborhood of $u$ (and $q(w)=c$) and $a$ in a neighborhood of $b$. In particular, if $u \in V$ satisfies $q(u)=c$ and $w$ is another vector in $V$ then the vector $tu-w+t\sigma_{c}$ tends, as $t\to\infty$, to $(\infty u)_{(u,w)}+\infty\sigma_{c}\in\partial\overline{V}_{\mathbb{F}}^{c}$. Note that if the topology on $\mathbb{F}$ is Hausdorff then $V_{\mathbb{F}}^{c}$ and $\mathbf{H}_{V,q}^{c}$ are open in $\overline{V}_{\mathbb{F}}^{c}$, and $\mathbf{H}_{V,q}^{c,o}$ is open in all three of them (as the intersection of these two sets).

We now extend Lemma \ref{normc} to the larger set $\mathbf{K}_{V,q}^{c}$ by considering also elements of $\partial\overline{V}_{\mathbb{F}}^{c}\subseteq\mathbf{H}_{V,q}^{c}$.
\begin{lem}
An element $w\in\mathbf{K}_{V,q}^{c}$ pairing with $e \in V_{U}$ to 0 is of the form $u+be$ for $u \in V$ with $q(u)=c$ and $b\in\mathbb{F}$, where for $c=0$ a vector $u \in V^{\perp}$ can match only with $b\neq0$. The correspondence matching our $w$ to the element $(\infty u)_{b}+\infty\sigma_{c}\in\overline{V}_{\mathbb{F}}^{c}$ is one-to-one between $\mathbf{K}_{V,q}^{c}\setminus\mathbf{K}_{V,q}^{c,o}$ and $\partial\overline{V}_{\mathbb{F}}^{c}$, which joins with the correspondence from Lemma \ref{normc} to a correspondence between $\mathbf{K}_{V,q}^{c}$ and $\mathbf{H}_{V,q}^{c}$. \label{bdnorm}
\end{lem}

\begin{proof}
An element $w \in V_{U}$ pairs to 0 with $e$ if and only if it lies in $V\oplus\mathbb{F}e$, and if $w$ is $u+be$ then $q_{U}(w)$, which must be $c$ for $w\in\mathbf{K}_{V,q}^{c}$, is $q(u)$. The fact that $V^{\perp}$ is excluded from $\mathbf{K}_{V,q}^{c}$ then easily gives the first assertion, and since the parameters in Definition \ref{boundary} are exactly the same, the remaining parts are proved like in Lemma \ref{normc}. This proves the lemma.
\end{proof}

This allows us to extend the definition of M\"{o}bius transformations to the case with a vanishing denominator.
\begin{prop}
Take $w$ and $z$ as in Proposition \ref{regMoeb}, as well as an element $\eta\in\Gamma^{\mathbb{F}^{\times}}(V_{U},q_{U})$ that is associated with the matrix $\binom{\alpha\ \ \beta}{\gamma\ \ \delta}\in\operatorname{V}(V,q)$ such that $N(\gamma z+\delta)=0$. Set $(\alpha z+\beta)(\gamma z+\delta)^{-1}$ to be $(\infty u)_{b}+\infty\sigma_{c}$ as in Definition \ref{boundary}, in which $u$ is the $V$-part of $(\alpha z+\beta)\overline{(\gamma z+\delta)}$ divided by $t\det\eta$, and $b=\frac{N(\alpha z+\beta)}{t\det\eta}$. Then this vector corresponds to the image of $w$ under $\pi_{U}(\eta)$, and the action of $\eta$ (or $\pi_{U}(\eta)$, or the matrix) is continuous on all of $\mathbf{H}_{V,q}^{c,o}$ in any of the topologies we considered on that space. \label{denom0}
\end{prop}

\begin{proof}
We saw in the proof of Proposition \ref{regMoeb} that with our notation $u$ we have $\pi_{U}(\eta)$ takes $w$ to $\frac{\eta w\eta^{*}}{\det\eta}$, and by combining with Remark \ref{expcal}, this vector becomes $u+\frac{N(\alpha z+\beta)}{t\det\eta}e+\frac{N(\gamma z+\delta)}{t\det\eta}f$. Moreover, Corollary \ref{MoebVq} determines the value of $q(u)$, which we write as $c-\frac{N(\gamma z+\delta)}{t\det\eta}\cdot\frac{N(\alpha z+\beta)}{t\det\eta}$. As we saw that for $N(\gamma z+\delta)\neq0$ the (finite) coefficient of $\sigma_{c}$ in $(\alpha z+\beta)(\gamma z+\delta)^{-1}$ was $\frac{t\det\eta}{N(\gamma z+\delta)}$, this becomes the value that we consider as $\infty$ when $N(\gamma z+\delta)=0$. Then the value of the extension of $q$ to $\infty u$ becomes $\infty^{2}c-\infty b$ with our value of $b$, which yields the continuity in each of our topologies. Since $\pi_{U}(\eta)(w)$ was evaluated as $u+be$ in our notation, which is the one associated with our element of $\partial\overline{V}_{\mathbb{F}}^{c}\subseteq\mathbf{H}_{V,q}^{c}$ via Lemma \ref{bdnorm}, and when $c=0$ the latter vector cannot be in $V^{\perp}$ since $\pi_{U}(\eta)^{-1}$ preserves $V^{\perp}$ and $w \not\in V^{\perp}$ (it pairs non-trivially with $e$), the remaining assertion follows as well. This proves the proposition.
\end{proof}

\smallskip

For extending the definition of M\"{o}bius transformations to an argument in $\partial\overline{V}_{\mathbb{F}}^{c}$, we require the following extension of Lemma \ref{Ncz+d}.
\begin{lem}
Take $z=(\infty u)_{b}+\infty\sigma_{c}\in\partial\overline{V}_{\mathbb{F}}^{c}$ and $\eta=\binom{\alpha\ \ \beta}{\gamma\ \ \delta}\in\operatorname{V}(V,q)$. Then the expressions $N(\gamma z+\delta)$, $N(\alpha z+\beta)$, and $(\alpha z+\beta)\overline{(\gamma z+\delta)}$ are of the form $N(\gamma z+\delta)_{*}\infty$, $N(\alpha z+\beta)_{*}\infty$, and $(\alpha z+\beta)\overline{(\gamma z+\delta)}_{*}\infty$ plus finite terms, where $N(\gamma z+\delta)_{*}$ and $N(\alpha z+\beta)_{*}$ are in $\mathbb{F}$ and $(\alpha z+\beta)\overline{(\gamma z+\delta)}_{*}\in\mathbf{H}_{V,q}^{c,o}$. Moreover, the $V$-part of $(\alpha z+\beta)\overline{(\gamma z+\delta)}_{*}$ has $q$-value $c\det^{2}\eta-N(\alpha z+\beta)_{*}N(\gamma z+\delta)_{*}$. \label{linininf}
\end{lem}

\begin{proof}
We substitute $v=\infty u$ and $t=\infty$ into the expressions from the proof of Lemma \ref{Ncz+d}. Then $N(\delta)$, $N(\beta)$ and $\beta\overline{\delta}$ are finite, and the extension of $q_{\mathbb{F}}^{c}$ to our value of $z$ is obtained by extending $q$ to $\infty u$ and observing that the natural quadratic value of $\infty\sigma_{c}$ is $\infty^{2}c$. As the resulting value of $q_{\mathbb{F}}^{c}(z)$ is $-\infty b$, we indeed obtain the desired linearity. More explicitly, we find that \[N(\gamma z+\delta)=\infty[bN(\gamma)+(\gamma u\overline{\delta}+\delta\overline{u}\overline{\gamma})],\quad N(\alpha z+\beta)=\infty[bN(\alpha)+(\alpha u\overline{\beta}+\beta\overline{u}\overline{\alpha})],\] and \[(\alpha z+\beta)\overline{(\gamma z+\delta)}=\infty[b\alpha\overline{\gamma}+(\alpha u\overline{\delta}+\beta\overline{u}\overline{\gamma})+(\alpha\delta^{*}-\beta\gamma^{*})\sigma_{c}],\] up to the finite terms, that we omit. This and the non-vanishing of $\alpha\delta^{*}-\beta\gamma^{*}$ yields the statements about linearity and the values of the expressions with the subscript $*$, and the evaluation of the $q$-value is obtained by multiplying the vector by $\infty$, arguing as in Corollary \ref{MoebVq}, and dividing the result by $\infty^{2}$. This proves the lemma.
\end{proof}

This allows us to give a well-defined formula for the action of a M\"{o}bius transformation on $z\in\partial\overline{V}_{\mathbb{F}}^{c}$, and extend Propositions \ref{regMoeb} and \ref{denom0} to this case.
\begin{prop}
Consider an element $z=(\infty u)_{b}+\infty\sigma_{c}\in\partial\overline{V}_{\mathbb{F}}^{c}$ and a matrix $\eta=\binom{\alpha\ \ \beta}{\gamma\ \ \delta}\in\operatorname{V}(V,q)$. Then if the expression $N(\gamma z+\delta)_{*}\in\mathbb{F}$ from Lemma \ref{linininf} does not vanish, then we define
\[(\alpha z+\beta)(\gamma z+\delta)^{-1}:=(\alpha z+\beta)\overline{(\gamma z+\delta)}_{*}\big/N(\gamma z+\delta)_{*}.\] When $N(\gamma z+\delta)_{*}=0$, we set $(\alpha z+\beta)(\gamma z+\delta)^{-1}$ to be $(\infty x)_{a}+\infty\sigma_{c}$, where $x$ is the $V$-part of $(\alpha z+\beta)\overline{(\gamma z+\delta)}_{*}$ divided by $\det\eta$, and $a=\frac{N(\alpha z+\beta)_{*}}{\det\eta}$. Then the action of $\eta$ preserves $\mathbf{H}_{V,q}^{c}$, and it takes the element of $\mathbf{K}_{V,q}^{c}$ that is associated with $z$ via Lemma \ref{bdnorm} to the one corresponding to $(\alpha z+\beta)(\gamma z+\delta)^{-1}$ in Lemma \ref{normc} or \ref{bdnorm}. \label{etaonbd}
\end{prop}

\begin{proof}
Following the proof of Proposition \ref{regMoeb}, we evaluate $\frac{\eta w\eta^{*}}{\det\eta}$ for $w$ being the vector $u+be$ that is associated with our $z$ via Lemma \ref{bdnorm}. Using Corollary \ref{VUact} again, the result is $\frac{1}{\alpha\delta^{*}-\beta\gamma^{*}}$ times \[[(\alpha u\overline{\delta}+\beta\overline{u}\overline{\gamma})+b\alpha\overline{\gamma}]+[(\alpha u\overline{\beta}+\beta\overline{u}\overline{\alpha})+bN(\alpha)]e+[(\gamma u\overline{\delta}+\delta\overline{u}\overline{\gamma})+bN(\gamma)]f,\] with the coefficients in the numerator being the $V$-part of $(\alpha z+\beta)\overline{(\gamma z+\delta)}_{*}$, $N(\alpha z+\beta)_{*}$, and $N(\gamma z+\delta)_{*}$ from Lemma \ref{linininf} respectively.

Now, when $N(\gamma z+\delta)_{*}\neq0$, the coefficient of $f$ is the (non-zero) inverse of $\frac{\alpha\delta^{*}-\beta\gamma^{*}}{N(\gamma z+\delta)_{*}}$, and as in the proof of Proposition \ref{regMoeb}, we see that $\frac{\eta w\eta^{*}}{\det\eta}$ is indeed associated with the asserted element of $\mathbf{H}_{V,q}^{c,o}$ via Lemma \ref{normc}. Moreover, we can write our expression for $\frac{\eta w\eta^{*}}{\det\eta}$ as $x+\frac{N(\alpha z+\beta)_{*}}{\det\eta}e+\frac{N(\gamma z+\delta)_{*}}{\det\eta}f$, with the last assertion of Lemma \ref{linininf} giving $q(x)=c-\frac{N(\gamma z+\delta)_{*}}{\det\eta}\cdot\frac{N(\alpha z+\beta)_{*}}{\det\eta}$. As in the proof of Proposition \ref{denom0}, in case $N(\gamma z+\delta)_{*}=0$, our value of $\infty$ is $\frac{\det\eta}{N(\gamma z+\delta)_{*}}$, the extension of $q$ takes $\infty x$ to $\infty^{2}c-\infty a$ for our $a$, and our value $x+ae$ of $\frac{\eta w\eta^{*}}{\det\eta}$ is indeed associated with $(\infty x)_{a}+\infty\sigma_{c}$. This proves the proposition.
\end{proof}

\smallskip

Now that the definition of the space $\mathbf{H}_{V,q}^{c}$ and the action of $\operatorname{V}(V,q)$ are defined in their entirety, we can consider some of their properties. Note that for $v \in V$, the matrix $\binom{1\ \ v}{0\ \ 1}$ lies in $\operatorname{V}(V,q)$, and, in fact, in $\operatorname{SV}(V,q)$. This can be easily seen using any of the conditions from Theorem \ref{Vahlen}, as the same argument proving that $v \in V$ with non-zero $q(v)$ satisfies $\pi(v)=r_{v}\in\mathcal{O}_{V^{\perp}}(V,q)$ shows that $V$ is contained in the set $\mathcal{T}(V,q)$ from Equation \eqref{entry} (and so are scalars from $\mathbf{F}$), and Remark \ref{invV}, or just Equation \eqref{pairC}, shows that the scalar involving $\alpha$ and $\beta$ in Condition 1 there is the negative of a pairing. Similarly, given $a\in\mathbb{F}^{\times}$ we have $\binom{a\ \ 0}{0\ \ 1}\in\operatorname{V}(V,q)$, and the group structure shows that these matrices combine to an embedding of the semi-direct product $\mathbb{F}^{\times}\rtimes(V,+)$ into $\operatorname{V}(V,q)$.

We can now establish a equivalence, already hinted about above.
\begin{lem}
Given a quadratic space $(V,q)$ over $\mathbb{F}$ and some $c\in\mathbb{F}$, the following conditions are equivalent: (1) There exists $v \in V$ with $q(v)=c$. (2) The inclusions $\mathbf{H}_{V,q}^{c,o}\subseteq\mathbf{H}_{V,q}^{c}$ and $\mathbf{K}_{V,q}^{c,o}\subseteq\mathbf{K}_{V,q}^{c}$ are proper. (3) The action of $\mathbb{F}^{\times}\rtimes(V,+)$ on $\mathbf{H}_{V,q}^{c}$ is not transitive. (4) The action of $\mathbb{F}^{\times}\rtimes(V,+)$ on $\mathbf{H}_{V,q}^{c}$ is not faithful. (5) For every $z\in\mathbf{H}_{V,q}^{c,o}$ there exists a matrix $\binom{\alpha\ \ \beta}{\gamma\ \ \delta}\in\operatorname{V}(V,q)$ such that the norm $N(\gamma z+\delta)$ from Lemma \ref{Ncz+d} vanishes. \label{FxVc}
\end{lem}

\begin{proof}
Condition (2) is equivalent to Definition \ref{boundary} being non-trivial, which happens exactly when Condition (1) is satisfied. Now, a general element of $\mathbb{F}^{\times}\rtimes(V,+)$ is a product $\binom{1\ \ v}{0\ \ 1}\binom{a\ \ 0}{0\ \ 1}=\binom{a\ \ v}{0\ \ 1}$, which takes the element $\sigma_{c}$ of $\mathbf{H}_{V,q}^{c,o}$ to $v+t\sigma_{c}$. This shows that $\mathbf{H}_{V,q}^{c,o}$ is the orbit of $\sigma_{c}$ under the action of this subgroup, and as the entries of $z=v+t\sigma_{c}$ determine the element of $\mathbb{F}^{\times}\rtimes(V,+)$ taking $\sigma_{c}$ to it, the action on this orbit is faithful. On the other hand, the formula from Proposition \ref{etaonbd} shows that in case an element $(\infty u)_{b}+\infty\sigma_{c}\in\partial\overline{V}_{\mathbb{F}}^{c}$ exists, the action of our element $\binom{a\ \ v}{0\ \ 1}$ takes it to $(\infty u)_{ab-(u,v)}+\infty\sigma_{c}$, and some non-trivial stabilizing elements exist. This means that Conditions (1) and (2) are equivalent to (3) and (4) as well. Finally, Proposition \ref{denom0} shows that Condition (5) implies Condition (2), and conversely, if Condition (1) gives a vector $u \in V$ with $q(u)=c$ and $z=v+t\sigma_{c}$ then the matrix with $\alpha=0$, $\gamma=-\beta=1$, and $\delta=tu-v$ is in $\operatorname{V}(V,q)$ by considerations similar to those for $\binom{a\ \ v}{0\ \ 1}$, and $N(\gamma z+\delta)=N(z-v+tu)=-q_{\mathbb{F}}^{c}(tu+t\sigma_{c})=0$. This proves the lemma.
\end{proof}

We deduce the following consequence.
\begin{prop}
The action of $\operatorname{V}(V,q)$ on $\mathbf{H}_{V,q}^{c}$ is transitive, and the stabilizer of the base point $\sigma_{c}\in\mathbf{H}_{V,q}^{c}$ consists of those matrices $\binom{\delta'\ \ -c\gamma'}{\gamma\ \ \ \ \ \delta\ }$ with $\gamma$ and $\delta$ in $\mathcal{T}(V,q)$ for which $\gamma\delta^{*} \in V$ and $N(\delta)+cN(\gamma)\neq0$. If the conditions from Lemma \ref{FxVc} hold then the action of $\operatorname{SV}(V,q)$ is also transitive, and in any case the stabilizer of $\sigma_{c}$ there consists of similar matrices, but with the equality $N(\delta)+cN(\gamma)=1$. \label{transact}
\end{prop}

\begin{proof}
The proof of Lemma \ref{FxVc} shows that all of $\mathbf{H}_{V,q}^{c,o}$ is contained in a single orbit of $\operatorname{V}(V,q)$, and if the conditions from Lemma \ref{FxVc} hold, then given an element $z=(\infty u)_{b}+\infty\sigma_{c}\in\partial\overline{V}_{\mathbb{F}}^{c}$, we can take a matrix in $\operatorname{V}(V,q)$ with $\alpha=0$, $\gamma=-\beta=1$, and $\delta \in V$. The fact that if $u \in V^{\perp}$ then $b\neq0$ implies that we can always choose $\delta$ such that the resulting expression $b-(u,\delta)$ for $N(\gamma z+\delta)_{*}$ from Lemma \ref{linininf} (using Equation \eqref{pairC} or Remark \ref{invV} again) will not vanish, and thus Proposition \ref{etaonbd} implies that this matrix connects also our element of $\partial\overline{V}_{\mathbb{F}}^{c}$ to our orbit. This implies the transitivity. For the stabilizer of $\sigma_{c}$, we get that $\binom{\alpha\ \ \beta}{\gamma\ \ \delta}$ stabilizes $\sigma_{c}$ if and only if $\alpha\sigma_{c}+\beta=\sigma_{c}(\gamma\sigma_{c}+\delta)$, which using the commutation relations between $\sigma_{c}$ and elements of $\mathcal{C}$ inside $\mathcal{C}(V_{\mathbb{F}}^{c},q_{\mathbb{F}}^{c})$, this is equivalent to $\alpha\sigma_{c}+\beta=\gamma'\sigma_{c}^{2}+\delta'\sigma_{c}=\delta'\sigma_{c}-c\gamma'$ (because $\sigma_{c}^{2}=-c$) and thus to $\alpha=\delta'$ and $\beta=-c\gamma'$. Since the determinant $\delta\alpha^{*}-\gamma\beta^{*}$ takes the value $N(\delta)+cN(\gamma)$, and we have the condition $\gamma\delta^{*} \in V$ from Condition (2) of Theorem \ref{Vahlen}, this determines the stabilizer.

Now, since $\det$ from Corollary \ref{detSV} is a group homomorphism, every element of $\operatorname{V}(V,q)$ is the product of an element of $\operatorname{SV}(V,q)$ and a matrix $\binom{a\ \ 0}{0\ \ 1}$ (with $a$ being the determinant), and as the latter matrix takes $\sigma_{c}$ to $a\sigma_{c}$, we deduce from the transitivity of the action of $\operatorname{V}(V,q)$ that $\mathbf{H}_{V,q}^{c}$ the union, over $a$, of the orbits of $a\sigma_{c}$ under $\operatorname{SV}(V,q)$. But we saw that the conditions from Lemma \ref{FxVc} imply that the quadratic space $(V_{\mathbb{F}}^{c},q_{\mathbb{F}}^{c})$ contains non-zero vectors having $q_{\mathbb{F}}^{c}$-value 0 that are not perpendicular to the entire space, so that the corollary to Proposition 3 in Chapter IV of \cite{[S]} implies that for every $a\in\mathbb{F}^{\times}$ there exists $\mu \in V_{\mathbb{F}}^{c}$ with $q_{\mathbb{F}}^{c}(\mu)=-\frac{1}{a}$. We can take that $\mu$ to be in $\mathbf{H}_{V,q}^{c,o}$ and after we write it as $v+d\sigma_{c}$ for $v \in V$ and $d\in\mathbb{F}^{\times}$, we find that the matrix $\binom{\alpha\ \ \beta}{\gamma\ \ \delta}$ with $\delta=v$, $\gamma=d$, $\alpha=-av$, and $\beta=-\frac{1+aq(v)}{d}$ lies in $\operatorname{SV}(V,q)$ (as the product of $\binom{1\ \ -av/d}{0\ \ \ \ \ 1\ \ }$ and $\binom{0\ \ -1/d}{d\ \ \ \ \ v\ \ }$) and satisfies $N(\gamma z+\delta)=N(v+d\sigma_{c})=-q_{\mathbb{F}}^{c}(\mu)=\frac{1}{a}$. As the value of $\beta$ equals $acd$ (since $\frac{1}{a}=-q_{\mathbb{F}}^{c}(v+d\sigma_{c})=cd^{2}-q(v)$), we deduce that this matrix takes $\sigma_{c}$ to $a\sigma_{c}$, so that all of these orbits of $\operatorname{SV}(V,q)$ are a single one. This proves the transitivity of the action of that group as well, and the determination of the stabilizer is the same as above, with the determinant from Corollary \ref{detSV} being equal to 1. This proves the proposition.
\end{proof}

The proof of Proposition \ref{transact} allows us to determine the orbits of $\operatorname{SV}(V,q)$ in $\mathbf{H}_{V,q}^{c}$ also when the conditions from Lemma \ref{FxVc} are not satisfied.
\begin{cor}
When $\mathbf{H}_{V,q}^{c}$ has no boundary points, the set of elements of $\mathbb{F}$ that can be obtained as $N(\gamma z+\delta)$ for $\binom{\alpha\ \ \beta}{\gamma\ \ \delta}$ in $\operatorname{SV}(V,q)$ (or in $\operatorname{V}(V,q)$) and $z\in\mathbf{H}_{V,q}^{c}$ is a subgroup of $\mathbb{F}^{\times}$, that contains the subgroup $(\mathbb{F}^{\times})^{2}$ of squares. The orbits of $\operatorname{SV}(V,q)$ in $\mathbf{H}_{V,q}^{c}$ are then in one-to-one correspondence with the cosets of that subgroup inside $\mathbb{F}^{\times}$. \label{SVorbits}
\end{cor}

\begin{proof}
We saw in Lemma \ref{FxVc} that in this case $N(\gamma z+\delta)$ never vanishes, i.e., lies in $\mathbb{F}^{\times}$, and since multiplication by $\binom{a\ \ 0}{0\ \ 1}$ from the left does not change the lower row, we see that these norms obtained from $\operatorname{SV}(V,q)$ and from $\operatorname{V}(V,q)$ are the same. Now, using the translation matrices $\binom{1\ \ v}{0\ \ 1}$ from $\operatorname{SV}(V,q)$, we deduce that two elements $v+t\sigma_{c}$ and $u+s\sigma_{c}$ are related via $\operatorname{SV}(V,q)$ if and only if $t\sigma_{c}$ and $s\sigma_{c}$ are, and we saw in the formula for the action from Proposition \ref{regMoeb} that this happens if and only if the ratio between $t$ and $s$ is some norm $N(\gamma z+\delta)$. The fact that these norms are a subgroup is now a consequence of the M\"{o}bius transformations defining an action, the squares are contained in it as the norms arising from $\binom{1/d\ \ 0}{\ 0\ \ \ d}\in\operatorname{SV}(V,q)$ for $d\in\mathbb{F}^{\times}$, and the relation between the $\operatorname{SV}(V,q)$-orbits and cosets of this subgroup, as subsets of $\mathbb{F}^{\times}$, also follows from this argument. This proves the corollary.
\end{proof}

\smallskip

Gathering all of these results yields the following theorem.
\begin{thm}
Take a quadratic space $(V,q)$ over $\mathbb{F}$, and a scalar $c$ from $\mathbb{F}$. Then the complement $\mathbf{H}_{V,q}^{c,o}$ of $V$ inside $V_{\mathbb{F}}^{c}$ admits a completion $\mathbf{H}_{V,q}^{c}$ on which the Vahlen group $\operatorname{V}(V,q)$ from Theorem \ref{Vahlen} acts transitively via M\"{o}bius transformations. This space is identified with the set $\mathbf{K}_{V,q}^{c}$ of vectors in the direct sum $V_{U}$ of $V$ with a hyperbolic plane that are not in $V^{\perp}$ and have $q_{U}$-value $c$ as sets with a $\operatorname{V}(V,q)$-action, so both are identified with the left coset space of the subgroup $\big\{\binom{\delta'\ \ -c\gamma'}{\gamma\ \ \ \ \ \delta\ }\;\big|\;\gamma,\delta\in\mathcal{T}(V,q),\ \gamma\delta^{*} \in V,\ N(\delta)+cN(\gamma)\neq0\big\}$ in $\operatorname{V}(V,q)$. For the special Vahlen group $\operatorname{SV}(V,q)$ from Corollary \ref{detSV}, the action is transitive when $\mathbf{H}_{V,q}^{c,o}\subsetneq\mathbf{H}_{V,q}^{c}$, and otherwise decomposes this space according to the cosets of a norm group inside $\mathbb{F}^{\times}$, with the stabilizer being the determinant 1 subgroup of the subgroup above. \label{thmVqc}
\end{thm}
Note that when $c=0$, the stabilizer from Proposition \ref{transact} and Theorem \ref{thmVqc} involves only $\delta\in\mathcal{T}(V,q)$ with $N(\delta)\neq0$, so that $\delta\in\Gamma^{\mathbb{F}^{\times}}(V,q)$, and this stabilizer is the semi-direct product $\Gamma^{\mathbb{F}^{\times}}(V,q)\rtimes(V,+)$. Conjugating by $\binom{0\ \ -1}{1\ \ \ \ 0}$ takes it to the stabilizer of $(\infty 0)_{1}+\infty\sigma_{0}$ (or of $e\in\mathbf{K}_{V,q}^{c}$), which is the (parabolic) upper-triangular subgroup of $\operatorname{V}(V,q)$. The intersection of both subgroups with $\operatorname{V}(V,q)$ restricts $\Gamma^{\mathbb{F}^{\times}}(V,q)$ to the group $\Gamma^{1}(V,q)$ from Corollary \ref{kerN}. However, when $c\neq0$ these subgroups are stabilizers of a vector of the non-vanishing $q_{U}$-value $c$, so they are isomorphic to the appropriate Clifford groups of the orthogonal complement, which are $\Gamma^{\mathbb{F}^{\times}}(V_{\mathbb{F}}^{-c},q_{\mathbb{F}}^{-c})$ and $\Gamma^{1}(V_{\mathbb{F}}^{-c},q_{\mathbb{F}}^{-c})$ respectively.

\smallskip

Recall that the hyperbolic space over $\mathbb{F}=\mathbb{R}$ also admits the bounded ball model. The corresponding situation is obtained when $V$ is definite, and the value of $c$ has the opposite sign. The ball model is obtained by a certain Cayley transform, but for making the image bounded, or at least contained in $V_{\mathbb{F}}^{c}$, on needs the denominator of the Cayley transform not to vanish. This can never be done for $\operatorname{V}(V,q)$, and also not for $\operatorname{SV}(V,q)$ when its action on $\mathbf{H}_{V,q}^{c}$ is still transitive. We shall therefore not discuss this model further in this paper.

\section{M\"{o}bius Transformations with Paravectors \label{ActPara}}

In this section we modify the arguments from the previous section in order to establish similar results for the Paravector Vahlen group $\widetilde{\operatorname{V}}(V,q)$, defined via the equivalent conditions from Theorem \ref{pvVah}. The results and arguments parallel very much those from the previous section, but as a lot of care has to be taken with additional details, we give the complete proofs here as well.

The Clifford algebra $\mathcal{C}_{U,\mathbb{F}}$ arising from the quadratic space $(V_{U,\mathbb{F}},q_{U,\mathbb{F}})$ is acted upon by its group of units $\mathcal{C}_{U}^{\times}$ by the same formula $\eta:\psi\to\eta\psi\eta^{*}$ for $\eta\in\mathcal{C}_{U,\mathbb{F}}^{\times}$ and $\psi\in\mathcal{C}_{U,\mathbb{F}}$. The vector space $V_{U,\mathbb{F}}$ is the direct sum of $V$, $\mathbb{F}\rho$, $\mathbb{F}e$, and $\mathbb{F}f$, and the first two summands form $V_{\mathbb{F}}$. We write $\mathcal{C}=\mathcal{C}(V,q)$, and get the following analogue of Lemma \ref{actVU}.
\begin{lem}
If $\eta\in\mathcal{C}_{U,\mathbb{F}}^{\times}$ is associated via Proposition \ref{isoUF} with the matrix $\binom{\alpha\ \ \beta}{\gamma\ \ \delta}$ with entries from $\mathcal{C}$, then the action of $\eta$ on $V_{U,\mathbb{F}}\subseteq\mathcal{C}_{U,\mathbb{F}}$ sends the elements $e$ and $f$ of $V_{U,\mathbb{F}}$ to the expressions $N(\alpha)_{\rho}e+N(\gamma)'_{\rho}f+(\alpha\overline{\gamma})_{\rho}e\rho f+(\gamma\overline{\alpha})'_{\rho}f\rho e$ and $N(\beta)_{\rho}e+N(\delta)'_{\rho}f+(\beta\overline{\delta})_{\rho}e\rho f+(\delta\overline{\beta})'_{\rho}f\rho e$ respectively. In addition, if $\xi\in\mathbb{F} \oplus V$ then its image in $V_{\mathbb{F}}$ under the isomorphism $\iota$ from Lemma \ref{relsefrho} is taken by the action of $\eta$ to \[(\alpha\xi\overline{\beta}+\beta\overline{\xi}\overline{\alpha})_{\rho}e+(\gamma\xi\overline{\delta}+\delta\overline{\xi}\overline{\gamma})'_{\rho}f+(\alpha\xi\overline{\delta}+\beta\overline{\xi}\overline{\gamma})_{\rho}e\rho f+(\gamma\xi\overline{\beta}+\delta\overline{\xi}\overline{\alpha})'_{\rho}f\rho e.\] \label{actpara}
\end{lem}

\begin{proof}
Proposition \ref{isoUF} gives the element $\alpha_{\rho}ef+\beta_{\rho}e\rho+\gamma'_{\rho}f\rho+\delta'_{\rho}fe$ of $\mathcal{C}_{U,\mathbb{F}}^{\times}$, and $\eta^{*}$ is given by $ef\delta^{*}_{\upsilon}+\rho e\beta^{*}_{\upsilon}+\rho f\gamma^{*}_{\upsilon}+fe\alpha^{*}_{\upsilon}$. Lemma \ref{relsefrho} transforms the latter expression into $ef\delta^{*}_{\rho}+\rho e\overline{\beta}_{\rho}+\rho f\gamma^{*}_{\rho}+fe\overline{\alpha}_{\rho}$, and using the vanishing of $e^{2}$ and $f^{2}$ as well as that of $e\rho e$ (since $e\rho=-\rho e$) and the equality $efe=e$ again, we obtain that $\eta e\eta^{*}$ reduces to  $\alpha_{\rho}e\overline{\alpha}_{\rho}+\alpha_{\rho}e\rho f\gamma^{*}_{\rho}+\gamma'_{\rho}f\rho e\overline{\alpha}_{\rho}+\gamma'_{\rho}f\rho e\rho f\gamma^{*}_{\rho}$. But we have $\rho e\rho=e$ (using $e\rho=-\rho e$ and $\rho^{2}=-1$), we saw that $fef=f$, Lemma \ref{relsefrho} gives the commutation relations between $e$, $f$, $\rho$, and elements of $\mathcal{C}_{\mathbb{F},+}$, and the map with index $\rho$ in Lemma \ref{isotoC+} is multiplicative, which transforms the latter expression to the desired one. Similarly, using $f\rho f=0$ as well (as $f\rho=-\rho f$), the expression for $\eta f\eta^{*}$ becomes $\beta_{\rho}e\rho f\rho e\overline{\beta}_{\rho}+\beta_{\rho}e\rho f\delta^{*}_{\rho}+\delta'_{\rho}f\rho e\overline{\beta}_{\rho}+\delta'_{\rho}f\delta^{*}_{\rho}$, which equals, by similar considerations and the equality $\rho f\rho=f$ (using $f\rho=-\rho f$ and $\rho^{2}=-1$), the asserted expression.

Take now some $\xi\in\mathbb{F} \oplus V$, with $\iota(\xi) \in V_{\mathbb{F}}\subseteq\mathcal{C}_{U,\mathbb{F}}$. The commutation relations between $\iota(\xi)$ and $e$, $f$, and $\rho$ in Lemma \ref{relsefrho}, the relation of the former vector to $\xi_{\rho}$, the anti-commutation of $\rho$ with $e$ and $f$, and the equalities $\rho e\rho=e$ and $\rho f\rho=f$ imply that $\eta\iota(\xi)$ equals \[\alpha_{\rho}\iota(\xi)ef+\beta_{\rho}\iota(\overline{\xi})e\rho+\gamma'_{\rho}\iota(\overline{\xi})f\rho+\delta'_{\rho}\iota(\xi)fe=\alpha_{\rho}\xi_{\rho}e\rho f+\beta_{\rho}\overline{\xi}_{\rho}e+\gamma'_{\rho}\overline{\xi}_{\rho}f+\delta'_{\rho}\xi_{\rho}f\rho e.\] Using the anti-commutation of $\rho$ with $e$ and $f$ and the vanishing of $e^{2}$ and $f^{2}$ again, with the equalities $efe=e$ and $fef=f$ as well as $\rho e\rho=e$ and $\rho f\rho=f$ once more, we find that the product of the latter expression with our formula for $\eta^{*}$ is \[\alpha_{\rho}\xi_{\rho}(e\rho f\delta^{*}+e\overline{\beta}_{\rho})+\beta_{\rho}\overline{\xi}_{\rho}(e\rho f\gamma^{*}_{\rho}+e\overline{\alpha}_{\rho})+\gamma'_{\rho}\overline{\xi}_{\rho}(f\delta^{*}_{\rho}+f\rho e\overline{\beta}_{\rho})+\delta'_{\rho}\xi_{\rho}(f\gamma^{*}_{\rho}+f\rho e\overline{\alpha}_{\rho}).\] Using the commutation relations from Lemma \ref{relsefrho}, gathering the coefficients of $e$, $f$, $e\rho f$, and $f\rho e$, noting that the index $\rho$ map from Lemma \ref{isotoC+} is a ring homomorphism, and recalling that for $\xi\in\mathbb{F} \oplus V$ we have $\xi'=\overline{\xi}$ and $\overline{\xi}'=\xi$ transforms the latter formula to the desired expression. This proves the lemma.
\end{proof}
Also here the formulae from Lemma \ref{actpara} are valid for all $\eta\in\mathcal{C}_{U,\mathbb{F}}$, but we only need it for $\eta$ in the paravector Vahlen group $\widetilde{\operatorname{V}}(V,q)$.
\begin{cor}
Assume that matrix $\binom{\alpha\ \ \beta}{\gamma\ \ \delta}$ from Lemma \ref{actpara} is in $\widetilde{\operatorname{V}}(V,q)$. Then the images of $e$, $f$, and $\iota(\xi)$ are $\iota(\alpha\overline{\gamma})+N(\alpha)e+N(\gamma)f$, $\iota(\beta\overline{\delta})+N(\beta)e+N(\delta)f$, and $\iota(\alpha\xi\overline{\delta}+\beta\overline{\xi}\overline{\gamma})+(\alpha\xi\overline{\beta}+\beta\overline{\xi}\overline{\alpha})e+(\gamma\xi\overline{\delta}+\delta\overline{\xi}\overline{\gamma})f$ respectively, all in $V_{U,\mathbb{F}}\subseteq\mathcal{C}_{U,\mathbb{F}}$. \label{paraact}
\end{cor}
Note that the three arguments of $\iota$ in Corollary \ref{paraact} are indeed in $\mathbb{F} \oplus V$ when the matrix is in $\widetilde{\operatorname{V}}(V,q)$, and this map can be applied to them.

\begin{proof}
The expressions $N(\alpha)$, $N(\beta)$, $\alpha\xi\overline{\beta}+\beta\overline{\xi}\overline{\alpha}$, $N(\gamma)$, $N(\delta)$, and $\gamma\xi\overline{\delta}+\delta\overline{\xi}\overline{\gamma}$, that one finds inside the coefficients of $e$ in $f$ in Lemma \ref{actpara}, are all in $\mathbb{F}$ by Condition 1 of Theorem \ref{pvVah}. Thus the coefficients of $f$ are invariant under the grading involution, and they are all equal to their images under the map with index $\rho$ from Lemma \ref{isotoC+}. The remaining terms are of the form $\psi_{\rho}e\rho f+\overline{\psi}'_{\rho}f\rho e$ for $\psi$ being $\alpha\overline{\gamma}$, $\beta\overline{\delta}$, or $\alpha u\overline{\delta}+\beta\overline{u}\overline{\gamma}$, all of which are in $\mathbb{F} \oplus V$. Thus $\overline{\psi}'=\psi$, and since Lemma \ref{relsefrho} expresses $\psi_{\rho}$ as $\iota(\psi)\rho$ (indeed, multiply the relation between them by $\rho$ from the right and recall that $\rho^{2}=-1$), we use the fact that $\rho e\rho=e$ and $\rho f\rho=f$ to see that $\iota(\psi)$ multiplies $ef+fe=1$. Thus indeed the images of $e$, $f$, and $\iota(\xi)$ are the asserted ones, and since to $\iota(\psi) \in V_{\mathbb{F}}$ we add scalar multiples of $e$ and $f$, these images are indeed in $V_{U,\mathbb{F}}$. This proves the corollary.
\end{proof}
Once again, the fact that the action of $\widetilde{\operatorname{V}}(V,q)$ preserves $V_{U,\mathbb{F}}$ is also a consequence of the isomorphism from Theorem \ref{pvVah} (indeed $\eta^{*}$ equals the determinant $\det\binom{\alpha\ \ \beta}{\gamma\ \ \delta}$ from Corollary \ref{SVpara} times $\eta'^{-1}$), via the definition of $\Gamma^{\mathbb{F}^{\times}}(V_{U,\mathbb{F}},q_{U,\mathbb{F}})_{+}$ in Corollary \ref{Normgrade} and Equation \eqref{twZClgrp}, but we shall require the explicit formulae from Corollary \ref{paraact}.

\smallskip

The group associated with $(V_{U,\mathbb{F}},q_{U,\mathbb{F}})$ in Equation \eqref{orthgrp} is $\operatorname{O}_{V^{\perp}}(V_{U,\mathbb{F}},q_{U,\mathbb{F}})$, because the space $V_{U,\mathbb{F}}^{\perp}$ from Equation \eqref{Vperp} is yet again $V^{\perp}$. Corollary \ref{Normgrade} shows that the group $\Gamma^{\mathbb{F}^{\times}}_{+}(V_{U,\mathbb{F}},q_{U,\mathbb{F}})$, and with it its isomorph $\widetilde{\operatorname{V}}(V,q)$ from Theorem \ref{pvVah}, surjects onto the associated special orthogonal group $\operatorname{SO}_{V^{\perp}}(V_{U,\mathbb{F}},q_{U,\mathbb{F}})$. Analogously to $\mathbf{K}_{V,q}^{c}$, we define $\widetilde{\mathbf{K}}_{V,q}^{c}$ to be those vectors $\omega \in V_{U,\mathbb{F}} \setminus V^{\perp}$ that satisfy $q_{U,\mathbb{F}}(\omega)=c$ (again, only for $c=0$ the condition $\omega \not\in V^{\perp}$ is meaningful). We recall the vector space $V_{\mathbb{F}}^{c}$ from before, and consider elements $\tau$ of $\mathbb{F} \oplus V_{\mathbb{F}}^{c}$, which we write as $\xi+t\sigma_{c}$ for $\xi\in\mathbb{F} \oplus V$ and $t\in\mathbb{F}$. The quadratic form $q_{\mathbb{F},\mathbb{F}}^{c}$ on $\mathbb{F} \oplus V_{\mathbb{F}}^{c}$ takes our $\tau$ to $q_{\mathbb{F}}(\xi)-ct^{2}$.

The correspondence between Zariski open subsets of $\widetilde{\mathbf{K}}_{V,q}^{c}$ and $\mathbb{F} \oplus V_{\mathbb{F}}^{c}$ is given in the following analogue of Lemma \ref{normc}.
\begin{lem}
The set $\widetilde{\mathbf{K}}_{V,q}^{c,o}$ of vectors $\omega\in\widetilde{\mathbf{K}}_{V,q}^{c}$ having non-zero pairing with $e$ in $V_{U,\mathbb{F}}$ are the vectors that can be written as $\omega=\frac{\iota(\xi)+f+(ct^{2}-q_{\mathbb{F}}(\xi))e}{t}$ with $\xi\in\mathbb{F} \oplus V$ and $t\in\mathbb{F}^{\times}$, where $\iota:\mathbb{F} \oplus V \to V_{\mathbb{F}}$ is the isomorphism from Lemma \ref{relsefrho}. Associating this vector with $\tau=\xi+t\sigma_{c}\in\mathbb{F} \oplus V_{\mathbb{F}}^{c}$ gives a one-to-one correspondence between $\widetilde{\mathbf{K}}_{V,q}^{c,o}$ and the complement $\widetilde{\mathbf{H}}_{V,q}^{c,o}$ of $\mathbb{F} \oplus V$ inside $\mathbb{F} \oplus V_{\mathbb{F}}^{c}$. \label{cnorm}
\end{lem}

\begin{proof}
The fact that $V_{U,\mathbb{F}}$ is the orthogonal direct sum of $U$ with the isomorph $V_{\mathbb{F}}$ of $\mathbb{F} \oplus V$ implies that if the non-zero pairing of an element $\omega$ there with $e$ is $\frac{1}{t}$ for $t\in\mathbb{F}^{\times}$ then $\omega$ is of the form $\frac{\iota(\xi)+f+se}{t}$ for $\xi\in\mathbb{F} \oplus V$ and $s\in\mathbb{F}$. Comparing the $q_{U,\mathbb{F}}$-image $\frac{q_{\mathbb{F}}(\xi)+s}{t^{2}}$ of $\omega$ with $c$ yields $s=ct^{2}-q_{\mathbb{F}}(\xi)$ as required. The bijectivity of the resulting correspondence between $\widetilde{\mathbf{K}}_{V,q}^{c,o}$ and $\widetilde{\mathbf{H}}_{V,q}^{c,o}$ is now clear, like in the proof of Lemma \ref{normc}. This proves the lemma.
\end{proof}

The M\"{o}bius transformation calculations here will require the following analogue of Lemma \ref{Ncz+d}.
\begin{lem}
Given $\tau\in\mathbb{F} \oplus V_{\mathbb{F}}^{c}$ and a matrix $\eta=\binom{\alpha\ \ \beta}{\gamma\ \ \delta}\in\widetilde{\operatorname{V}}(V,q)$, the two norms $N(\gamma\tau+\delta)$ and $N(\alpha\tau+\beta)$ lie in $\mathbb{F}$, and $(\alpha\tau+\beta)\overline{(\gamma\tau+\delta)}$ is in $\mathbb{F} \oplus V_{\mathbb{F}}^{c}$. In the case where $\tau$ is in $\widetilde{\mathbf{H}}_{V,q}^{c,o}$, the latter vector lies there as well. \label{Nctau+d}
\end{lem}

\begin{proof}
We write $\tau=\xi+t\sigma_{c}$, and recall the commutation relations of $\sigma_{c}$ with elements of $\mathcal{C}$ inside $\mathcal{C}(V_{\mathbb{F}}^{c},q_{\mathbb{F}}^{c})$ and its image under Clifford involution. This, together with the fact that for $\tau\in\mathbb{F} \oplus V_{\mathbb{F}}^{c}$ we have that $N(\tau)=-q_{\mathbb{F},\mathbb{F}}^{c}(\tau)\in\mathbb{F}$ is central in $\mathcal{C}(V_{\mathbb{F}}^{c},q_{\mathbb{F}}^{c})$ and $\gamma\delta^{*}=\delta\gamma^{*}$ for $\eta\in\widetilde{\operatorname{V}}(V,q)$ via Condition 1 of Theorem \ref{pvVah}, allows us to write $N(\gamma\tau+\delta)=(\gamma\tau+\delta)\overline{(\gamma\tau+\delta)}$ as \[\gamma N(\tau)\overline{\gamma}+\gamma\xi\overline{\delta}+\delta\overline{\xi}\overline{\gamma}+t(\gamma\delta^{*}-\delta\gamma^{*})\sigma_{c}+\delta\overline{\delta}=-q_{\mathbb{F},\mathbb{F}}^{c}(\tau)N(\gamma)+ (\gamma\xi\overline{\delta}+\delta\overline{\xi}\overline{\gamma})+N(\delta),\] which is in $\mathbb{F}$ by Condition 1 of Theorem \ref{pvVah}. A similar argument shows that \[N(\alpha\tau+\beta)=-q_{\mathbb{F},\mathbb{F}}^{c}(\tau)N(\alpha)+(\alpha\xi\overline{\beta}+\beta\overline{\xi}\overline{\alpha})+N(\beta)\in\mathbb{F},\] and we also obtain that \[(\alpha\tau+\beta)\overline{(\gamma\tau+\delta)}=-q_{\mathbb{F},\mathbb{F}}^{c}(\tau)\alpha\overline{\gamma}+(\alpha\xi\overline{\delta}+\beta\overline{\xi}\overline{\gamma})+t(\alpha\delta^{*}-\beta\gamma^{*})\sigma_{c} +\beta\overline{\delta},\] with the third one being a non-zero scalar multiple of $t\sigma_{c}$ and the remaining terms lying in $\mathbb{F} \oplus V$ by Theorem \ref{pvVah}. This proves the lemma.
\end{proof}

\begin{cor}
For a matrix $\eta$ and an element $\tau=\xi+t\sigma_{c}$ as in Lemma \ref{cnorm}, the $\mathbb{F} \oplus V$-part of the paravector $(\alpha\tau+\beta)\overline{(\gamma\tau+\delta)}\in\mathbb{F} \oplus V_{\mathbb{F}}^{c}$ has $q_{\mathbb{F}}$-value $ct^{2}\det^{2}\eta-N(\alpha\tau+\beta)N(\gamma\tau+\delta)$, where $\det\eta$ is the scalar $\alpha\delta^{2}-\beta\gamma^{*}\in\mathbb{F}^{\times}$ from Theorem \ref{pvVah} and Corollary \ref{SVpara}. \label{VqFMoeb}
\end{cor}

\begin{proof}
As in the proof of Corollary \ref{MoebVq}, paravectors share the property of vectors that their quadratic values are minus their norms, which shows (by invoking Remark \ref{invV} again) that $q_{\mathbb{F},\mathbb{F}}^{c}\big((\alpha\tau+\beta)\overline{(\gamma\tau+\delta)}\big)=-N(\alpha\tau+\beta)N(\gamma\tau+\delta)$. Writing this paravector vector, via the proof of Lemma \ref{cnorm}, as its $\mathbb{F} \oplus V$-part plus $t\det\eta\cdot\sigma_{c}$, we deduce that the $q_{\mathbb{F}}$-value (or equivalently the $q_{\mathbb{F},\mathbb{F}}^{c}$-value) of the former, is difference between that of $(\alpha\tau+\beta)\overline{(\gamma\tau+\delta)}$ and the $q_{\mathbb{F},\mathbb{F}}^{c}$-value $ct^{2}\det^{2}\eta$ of $t\det\eta\cdot\sigma_{c}$ (as $q_{\mathbb{F},\mathbb{F}}^{c}(\sigma_{c})=q_{\mathbb{F}}^{c}(\sigma_{c})=-c$). This proves the corollary.
\end{proof}

The regular part of the M\"{o}bius transformation formula here is as follows.
\begin{prop}
Take $\omega\in\widetilde{\mathbf{K}}_{V,q}^{c,o}$ and $\tau\in\mathbf{\widetilde{H}}_{V,q}^{c,o}$ that are related by Lemma \ref{cnorm}, as well as a matrix $\binom{\alpha\ \ \beta}{\gamma\ \ \delta}\in\widetilde{\operatorname{V}}(V,q)$ for which the norm $N(\gamma z+\delta)$ from Lemma \ref{Nctau+d} is non-zero, and let $\eta\in\Gamma^{\mathbb{F}^{\times}}(V_{U,\mathbb{F}},q_{U,\mathbb{F}})_{+}$ be the element corresponding to this matrix in Theorem \ref{pvVah}. Then if $\pi_{U,\mathbb{F}}$ is the map from Theorem \ref{GammaO} that is associated with $(V_{U,\mathbb{F}},q_{U,\mathbb{F}})$, then $\pi_{U,\mathbb{F}}(\eta)\in\operatorname{SO}_{V^{\perp}}(V_{U,\mathbb{F}},q_{U,\mathbb{F}})$ takes $\omega$ to the element of $\mathbf{K}_{V,q}^{c,o}$ that is associated with $(\alpha\tau+\beta)(\gamma\tau+\delta)^{-1}\in\mathbf{H}_{V,q}^{c,o}$. \label{Moebreg}
\end{prop}

\begin{proof}
Like in the proof of Proposition \ref{regMoeb}, we get $\pi_{U,\mathbb{F}}(\eta)(\omega)=\eta\omega\eta'^{-1}=\frac{\eta\omega\eta^{*}}{\det\eta}$ with $\det\eta$ being $\alpha\delta^{*}-\beta\gamma^{*}$ as in Theorem \ref{pvVah} and Corollary \ref{SVpara}, and with $\omega$ as in Lemma \ref{cnorm} this vector is $\eta\big[\iota(\xi)+f+\big(ct^{2}-q_{\mathbb{F}}(\xi)\big)e\big]\eta^{*}$ divided by the scalar $t\det\eta=t(\alpha\delta^{*}-\beta\gamma^{*})$. We evaluate the numerator, via Corollary \ref{paraact}, as \[\iota\big[(\alpha\xi\overline{\delta}+\beta\overline{\xi}\overline{\gamma})+\beta\overline{\delta}+\big(ct^{2}-q_{\mathbb{F}}(\xi)\big)\alpha\overline{\gamma}\big]+\big[(\alpha\xi\overline{\beta}+ \beta\overline{\xi}\overline{\alpha})+N(\beta)+\big(ct^{2}-q_{\mathbb{F}}(\xi)\big)N(\alpha)\big]e+\] \[+\big[(\gamma\xi\overline{\delta}+\delta\overline{\xi}\overline{\gamma})+N(\delta)+\big(ct^{2}-q_{\mathbb{F}}(\xi)\big)N(\gamma)\big]f,\] where the coefficient $ct^{2}-q_{\mathbb{F}}(\xi)$ can also be written as $-q_{\mathbb{F},\mathbb{F}}^{c}(\tau)$ for the associated element $\tau=\xi+t\sigma_{c}\in\mathbf{\widetilde{H}}_{V,q}^{c,o}$ from Lemma \ref{cnorm}. Using Lemma \ref{Nctau+d}, the coefficient of $f$ here is $N(\gamma\tau+\delta)$, and the $V_{\mathbb{F}}$-part of our expression is the $\iota$-image of the $\mathbb{F} \oplus V$-part of $(\alpha\tau+\beta)\overline{(\gamma\tau+\delta)}$.

It follows that the image of $\omega\in\widetilde{\mathbf{K}}_{V,q}^{c,o}$ under $\pi_{U,\mathbb{F}}(\eta)$ is in $\widetilde{\mathbf{K}}_{V,q}^{c,o}$, with the inverse of the coefficient of $f$ being $\frac{t(\alpha\delta^{*}-\beta\gamma^{*})}{N(\gamma\tau+\delta)}$, and multiplying the $V_{\mathbb{F}}$-part part of $\pi_{U,\mathbb{F}}(\eta)(\omega)$ by this scalar gives $\frac{\iota[(\alpha\tau+\beta)\overline{(\gamma\tau+\delta)}-t(\alpha\delta^{*}-\beta\gamma^{*})\sigma_{c}]}{N(\gamma\tau+\delta)}$. But our M\"{o}bius expression $(\alpha\tau+\beta)(\gamma\tau+\delta)^{-1}\in\mathbf{H}_{V,q}^{c,o}$ is $(\alpha\tau+\beta)\overline{(\gamma\tau+\delta)}$ divided by $N(\gamma\tau+\delta)$, in which the coefficient of $\sigma_{c}$ is the former quotient, and whose $\mathbb{F} \oplus V$-part becomes, after applying $\iota$, the latter quotient, making it the element corresponding via Lemma \ref{cnorm} to $\pi_{U,\mathbb{F}}(\eta)(\omega)$. This proves the proposition.
\end{proof}

\begin{rmk}
Similarly to Remark \ref{expcal}, the coefficient of $e$ in the numerator in the proof of Proposition \ref{Moebreg} is $N(\alpha\tau+\beta)$, Corollary \ref{VqFMoeb} implies that the $q_{U,\mathbb{F}}$-value of the numerator is $ct^{2}\det^{2}\eta$ without the applying orthogonality of $\pi_{U,\mathbb{F}}(\eta)$, and Theorem \ref{pvVah} and Remark \ref{invV} can give this value by direct calculations. \label{calexp}
\end{rmk}

\smallskip

Also here it may happen that for some $c\neq0$, there is no vector in $V_{\mathbb{F}}$ having $q_{\mathbb{F}}$-value $c$, and then $\widetilde{\mathbf{K}}_{V,q}^{c}=\widetilde{\mathbf{K}}_{V,q}^{c,o}$ maps bijectively to $\widetilde{\mathbf{K}}_{V,q}^{c,o}$ via Lemma \ref{cnorm}, and Proposition \ref{Moebreg} already describes the entire action of $\operatorname{SO}_{V^{\perp}}(V_{U,\mathbb{F}},q_{U,\mathbb{F}})$ (or $\Gamma^{\mathbb{F}^{\times}}(V_{U,\mathbb{F}},q_{U,\mathbb{F}})_{+}$, or $\widetilde{\operatorname{V}}(V,q)$). This happens, for example, if $\mathbb{F}=\mathbb{R}$, $(V,q)$ is negative definite, and $c=1$, where the restriction of our Proposition \ref{Moebreg} to $\widetilde{\operatorname{SV}}(V,q)$ reproduces Propositions 5.1 and 5.3 of \cite{[EGM]}. But if $V_{\mathbb{F}}$ represents $c$ then $\widetilde{\mathbf{K}}_{V,q}^{c,o}\subsetneq\widetilde{\mathbf{K}}_{V,q}^{c}$, and we have to complete $\widetilde{\mathbf{H}}_{V,q}^{c,o}$ by adding appropriate boundary points to $\mathbb{F} \oplus V_{\mathbb{F}}^{c}$ and include them for obtaining the space $\widetilde{\mathbf{H}}_{V,q}^{c}$. This is carried out in the following analogue of Definition \ref{boundary}.
\begin{defn}
Let $\lambda$ be a paravector in $\mathbb{F} \oplus V$ with $q_{\mathbb{F}}(\lambda)=c$, and take $b\in\mathbb{F}$, such that if $c=0$ and $\xi \in V^{\perp}$ then $b\neq0$. Then the corresponding \emph{boundary point} $(\infty\lambda)_{b}+\infty\sigma_{c}$ of $\mathbb{F} \oplus V_{\mathbb{F}}^{c}$ satisfies the property that the value of $\infty\lambda$ under the extension of $q_{\mathbb{F}}$ is $\infty^{2}c-\infty b$. We denote the set of all boundary point of $\mathbb{F} \oplus V_{\mathbb{F}}^{c}$ by $\partial(\mathbb{F} \oplus V_{\mathbb{F}}^{c})$, the completion of $\mathbb{F} \oplus V_{\mathbb{F}}^{c}$ by adding these boundary points by $\overline{\mathbb{F} \oplus V}_{\mathbb{F}}^{c}$, and the complement $\widetilde{\mathbf{H}}_{V,q}^{c,o}\cup\partial(\mathbb{F} \oplus V_{\mathbb{F}}^{c})$ of $\mathbb{F} \oplus V$ inside $\overline{\mathbb{F} \oplus V}_{\mathbb{F}}^{c}$ by $\widetilde{\mathbf{H}}_{V,q}^{c}$. \label{bdpara}
\end{defn}
Also here any topology on $\mathbb{F} \oplus V_{\mathbb{F}}^{c}$ extends naturally to $\overline{\mathbb{F} \oplus V}_{\mathbb{F}}^{c}$, with the points near a boundary element $(\infty\lambda)_{b}+\infty\sigma_{c}$ being those $\tau=\xi+t\sigma_{c}$ with $t$ near $\infty$ in $\mathbb{P}^{1}(\mathbb{F})$ and the pairing $(t\lambda-\xi,\lambda)_{\mathbb{F}}$ from Equation \eqref{parabil} being near $b$, as well as those $(\infty\mu)_{a}+\infty\sigma_{c}$ in which $\mu\in\mathbb{F} \oplus V$ is near $\lambda$ and satisfies $q_{\mathbb{F}}(\mu)=c$, and $a$ is near $b$. The fact that for $\lambda$ and $\xi$ $\mathbb{F} \oplus V$ with $q_{\mathbb{F}}(\mu)=c$, the limit of $t\lambda-\xi+t\sigma_{c}$ as $t\to\infty$ is $(\infty\lambda)_{(\lambda,\xi)_{\mathbb{F}}}+\infty\sigma_{c}\in\partial(\mathbb{F} \oplus V_{\mathbb{F}}^{c})$ is valid here as well, and so are the openness of $\mathbb{F} \oplus V_{\mathbb{F}}^{c}$ and $\widetilde{\mathbf{H}}_{V,q}^{c}$ in $\overline{\mathbb{F} \oplus V}_{\mathbb{F}}^{c}$ and that of $\widetilde{\mathbf{H}}_{V,q}^{c,o}$ in the latter three sets when $\mathbb{F}$ carries a Hausdorff topology.

The extension of Lemma \ref{cnorm} to $\widetilde{\mathbf{K}}_{V,q}^{c}$ by mapping the elements in the complement of $\widetilde{\mathbf{K}}_{V,q}^{c,o}$ to $\partial(\mathbb{F} \oplus V_{\mathbb{F}}^{c})\subseteq\widetilde{\mathbf{H}}_{V,q}^{c}$ is as follows.
\begin{lem}
If $\omega\in\widetilde{\mathbf{K}}_{V,q}^{c}$ is orthogonal to $e$ in $V_{U,\mathbb{F}}$ then it equals $\iota(\lambda)+be$ for a paravector $\lambda\in\mathbb{F} \oplus V$ with $q_{\mathbb{F}}(\lambda)=c$ and $b\in\mathbb{F}$, where if $c=0$ and $\lambda \in V^{\perp}$ then $b\neq0$. Associating our $\omega$ with the element $(\infty\lambda)_{b}+\infty\sigma_{c}\in\partial(\mathbb{F} \oplus V_{\mathbb{F}}^{c})$ yields a one-to-one between $\widetilde{\mathbf{K}}_{V,q}^{c}\setminus\widetilde{\mathbf{K}}_{V,q}^{c,o}$ and $\partial(\mathbb{F} \oplus V_{\mathbb{F}}^{c})$, and with Lemma \ref{cnorm} we get a one-to-one correspondence between $\widetilde{\mathbf{K}}_{V,q}^{c}$ and $\widetilde{\mathbf{H}}_{V,q}^{c}$. \label{normbd}
\end{lem}

\begin{proof}
The elements of $V_{U,\mathbb{F}}$ that are orthogonal to $e$ are precisely those in the direct sum of $V_{\mathbb{F}}=\iota(\mathbb{F} \oplus V)$ and $\mathbb{F}e$, and for $\omega=\iota(\lambda)+be$ in that space we have $q_{U,\mathbb{F}}(\omega)=q_{\mathbb{F}}(\lambda)$. Thus $\omega\in\widetilde{\mathbf{K}}_{V,q}^{c}$ if and only if $q_{\mathbb{F}}(\lambda)=c$, and excluding $V^{\perp}$ (with $b=0$) for $c=0$ yields the first assertion. The remaining parts are proved like in Lemma \ref{bdnorm}, using the parameters from Definition \ref{bdpara}. This proves the lemma.
\end{proof}

We can now consider paravector M\"{o}bius transformations with a vanishing denominator.
\begin{prop}
Consider $\omega\in\widetilde{\mathbf{K}}_{V,q}^{c,o}$, with associated $\tau\in\widetilde{\mathbf{H}}_{V,q}^{c,o}$ via Lemma \ref{cnorm}, and take $\eta\in\Gamma^{\mathbb{F}^{\times}}(V_{U,\mathbb{F}},q_{U,\mathbb{F}})$ such that the associated matrix $\binom{\alpha\ \ \beta}{\gamma\ \ \delta}$ in $\widetilde{\operatorname{V}}(V,q)$ satisfies $N(\gamma\tau+\delta)=0$. We define $(\alpha\tau+\beta)(\gamma\tau+\delta)^{-1}$ to be the element $(\infty\lambda)_{b}+\infty\sigma_{c}$ from Definition \ref{bdpara}, where $\lambda$ is the $\mathbb{F} \oplus V$-part of $(\alpha\tau+\beta)\overline{(\gamma\tau+\delta)}$ divided by $t\det\eta$, and $b$ equals $\frac{N(\alpha\tau+\beta)}{t\det\eta}$. Then Lemma \ref{normbd} associates this vector to $\pi_{U,\mathbb{F}}(\eta)(\omega)$, and the resulting extension of the M\"{o}bius transformation is continuous on $\widetilde{\mathbf{H}}_{V,q}^{c,o}$ wherever this notion makes sense. \label{denomvan}
\end{prop}

\begin{proof}
Using the notation $\lambda$, the proof of Proposition \ref{Moebreg} gives, via Remark \ref{calexp}, that $\pi_{U,\mathbb{F}}(\eta)(\omega)=\frac{\eta\omega\eta^{*}}{\det\eta}$ equals $\iota(\lambda)+\frac{N(\alpha\tau+\beta)}{t\det\eta}e+\frac{N(\gamma\tau+\delta)}{t\det\eta}f$, and Corollary \ref{VqFMoeb} yields $q_{\mathbb{F}}(\lambda)=c-\frac{N(\gamma\tau+\delta)}{t\det\eta}\cdot\frac{N(\alpha\tau+\beta)}{t\det\eta}$. The coefficient of $\sigma_{c}$ in $(\alpha\tau+\beta)(\gamma\tau+\delta)^{-1}$ was $\frac{t\det\eta}{N(\gamma\tau+\delta)}$ when the denominator was non-zero, so this would be $\infty$ in case $N(\gamma\tau+\delta)$, and the extension of $q_{\mathbb{F}}$ takes $\infty\lambda$ to $\infty^{2}c-\infty b$ with the asserted $b$, giving continuity again. The equality $\pi_{U,\mathbb{F}}(\eta)(\omega)=\iota(\lambda)+be$, Lemma \ref{normbd}, and the exclusion of $V^{\perp}$ when $c=0$ yield the rest like in the proof of Proposition \ref{denom0}. This proves the proposition.
\end{proof}

\smallskip

The linearity in $\infty$ on $\partial(\mathbb{F} \oplus V_{\mathbb{F}}^{c})$, like in Lemma \ref{linininf}, holds here as well.
\begin{lem}
Given $\omega=(\infty\lambda)_{b}+\infty\sigma_{c}\in\partial(\mathbb{F} \oplus V_{\mathbb{F}}^{c})$ and $\eta=\binom{\alpha\ \ \beta}{\gamma\ \ \delta}\in\widetilde{\operatorname{V}}(V,q)$, the expressions $N(\gamma\tau+\delta)$, $N(\alpha\tau+\beta)$, and $(\alpha\tau+\beta)\overline{(\gamma\tau+\delta)}$ are linear at $\infty$, i.e., up to finite terms they are given by $N(\gamma\tau+\delta)_{*}\infty$, $N(\alpha\tau+\beta)_{*}\infty$, and $(\alpha\tau+\beta)\overline{(\gamma\tau+\delta)}_{*}\infty$, in which $N(\gamma\tau+\delta)_{*}$ and $N(\alpha\tau+\beta)_{*}$ lie in $\mathbb{F}$ and $(\alpha\tau+\beta)\overline{(\gamma\tau+\delta)}_{*}$ is in $\widetilde{\mathbf{H}}_{V,q}^{c,o}$. In addition, the $q_{\mathbb{F}}$-value of the $\mathbb{F} \oplus V$-part of $(\alpha\tau+\beta)\overline{(\gamma\tau+\delta)}_{*}$ is $c\det^{2}\eta-N(\alpha\tau+\beta)_{*}N(\gamma\tau+\delta)_{*}$. \label{inflin}
\end{lem}

\begin{proof}
Considering the formulae from the proof of Lemma \ref{Nctau+d}, but with $\tau=\infty\xi$ and $t=\infty$, we ignore the finite terms $N(\delta)$, $N(\beta)$ and $\beta\overline{\delta}$, and the terms involving $\xi$ or $\overline{\xi}$ are linear at $\infty$. The extension of $q_{\mathbb{F},\mathbb{F}}^{c}$ take $\tau$ to the difference between between the value $\infty^{2}c-\infty b$ of the extension of $q_{\mathbb{F}}$ to $\infty\xi$ as in Definition \ref{bdpara} and the quadratic value $\infty^{2}c$ of $\infty\sigma_{c}$ under the extension of $q_{\mathbb{F},\mathbb{F}}^{c}$, which is the linear term $-\infty b$. This gives the required result, with \[N(\gamma\tau+\delta)_{*}=bN(\gamma)+(\gamma\xi\overline{\delta}+\delta\overline{\xi}\overline{\gamma})\in\mathbb{F},\quad N(\alpha\tau+\beta)_{*}=bN(\alpha)+(\alpha\xi\overline{\beta}+\beta\overline{\xi}\overline{\alpha})\in\mathbb{F},\] and \[(\alpha\tau+\beta)\overline{(\gamma\tau+\delta)}_{*}=b\alpha\overline{\gamma}+(\alpha\xi\overline{\delta}+\beta\overline{\xi}\overline{\gamma})+(\alpha\delta^{*}-\beta\gamma^{*})\sigma_{c}\in\mathbb{F} \oplus V_{\mathbb{F}}^{c},\] the latter is in $\widetilde{\mathbf{H}}_{V,q}^{c,o}$ due to the non-vanishing of $\alpha\delta^{*}-\beta\gamma^{*}$, and the same argument from the proof of Lemma \ref{linininf}, but now with Corollary \ref{VqFMoeb}, yields the value of the $\mathbb{F} \oplus V$-part of the latter paravector under $q_{\mathbb{F}}$. This proves the lemma.
\end{proof}

The following analogue of Proposition \ref{etaonbd} extends the action of the paravector M\"{o}bius transformations from Propositions \ref{Moebreg} and \ref{denomvan} to $\partial(\mathbb{F} \oplus V_{\mathbb{F}}^{c})$, and thus to all of $\widetilde{\mathbf{H}}_{V,q}^{c}$.
\begin{prop}
For $\tau=(\infty\xi)_{b}+\infty\sigma_{c}\in\partial(\mathbb{F} \oplus V_{\mathbb{F}}^{c})$ and $\eta=\binom{\alpha\ \ \beta}{\gamma\ \ \delta}\in\widetilde{\operatorname{V}}(V,q)$, consider the expression $N(\gamma\tau+\delta)_{*}\in\mathbb{F}$ from Lemma \ref{linininf}. If it is non-zero, then we set \[(\alpha\tau+\beta)(\gamma\tau+\delta)^{-1}:=(\alpha\tau+\beta)\overline{(\gamma\tau+\delta)}_{*}\big/N(\gamma\tau+\delta)_{*}.\] In case it does vanish, we define $(\alpha\tau+\beta)(\gamma\tau+\delta)^{-1}$ to be $(\infty\mu)_{a}+\infty\sigma_{c}$, in which $\mu$ is $\frac{1}{\det\eta}$ times the $\mathbb{F} \oplus V$-part of $(\alpha\tau+\beta)\overline{(\gamma\tau+\delta)}_{*}$, and $a$ equals $\frac{N(\alpha\tau+\beta)_{*}}{\det\eta}$. This gives an action of $\eta$ on $\widetilde{\mathbf{H}}_{V,q}^{c}$, such that the element of $\widetilde{\mathbf{K}}_{V,q}^{c}$ that corresponds to $\tau$ via in \ref{normbd} is sent to the one associated with $(\alpha\tau+\beta)(\gamma\tau+\delta)^{-1}$ via Lemma \ref{cnorm} or \ref{normbd}. \label{bdeta}
\end{prop}

\begin{proof}
Lemma \ref{normbd} associates with our $\tau$ the vector $\omega=\iota(\lambda)+be \in V_{U,\mathbb{F}}$, and Corollary \ref{paraact} evaluates $\frac{\eta\omega\eta^{*}}{\det\eta}$ for this $\omega$ as $\frac{1}{\alpha\delta^{*}-\beta\gamma^{*}}$ times \[[(\alpha\xi\overline{\delta}+\beta\overline{\xi}\overline{\gamma})+b\alpha\overline{\gamma}]+[(\alpha\xi\overline{\beta}+\beta\overline{\xi}\overline{\alpha})+bN(\alpha)]e+[(\gamma\xi\overline{\delta}+ \delta\overline{\xi}\overline{\gamma})+bN(\gamma)]f.\] In the terminology of the proof of Lemma \ref{inflin}, these coefficients are $\mathbb{F} \oplus V$-part of $(\alpha\tau+\beta)\overline{(\gamma\tau+\delta)}_{*}$, $N(\alpha\tau+\beta)_{*}$, and $N(\gamma\tau+\delta)_{*}$ respectively. If $N(\gamma\tau+\delta)_{*}\neq0$, then the proof of Proposition \ref{Moebreg} shows that $\frac{\eta\omega\eta^{*}}{\det\eta}$ is the vector corresponding to the desired element of $\widetilde{\mathbf{H}}_{V,q}^{c,o}$ via Lemma \ref{normc}, in which the coefficient multiplying $\sigma_{c}$ is $\frac{\alpha\delta^{*}-\beta\gamma^{*}}{N(\gamma\tau+\delta)_{*}}$. Using the paravector $\mu$, and Lemma \ref{inflin}, the total expression for $\frac{\eta\omega\eta^{*}}{\det\eta}$ is $\iota(\mu)+\frac{N(\alpha\tau+\beta)_{*}}{\det\eta}e+\frac{N(\gamma\tau+\delta)_{*}}{\det\eta}f$ with $q_{\mathbb{F}}(\mu)=c-\frac{N(\gamma\tau+\delta)_{*}}{\det\eta}\cdot\frac{N(\alpha\tau+\beta)_{*}}{\det\eta}$, so that when $N(\gamma\tau+\delta)_{*}=0$, the value of $\infty$ is $\frac{\det\eta}{N(\gamma\tau+\delta)_{*}}$, the total vector is $\iota(\mu)+ae$, and the extension of $q_{\mathbb{F}}$ takes $\infty\mu$ to $\infty^{2}c-\infty a$, as required by Lemma \ref{normbd}. This proves the proposition.
\end{proof}

\smallskip

For the properties of the action of $\widetilde{\operatorname{V}}(V,q)$ on $\widetilde{\mathbf{H}}_{V,q}^{c}$, we observe that $\mathbf{F} \oplus V$ is contained in the set $\widetilde{\mathcal{T}}(V,q)$ from Equation \eqref{enpara}, and using any of the conditions of Theorem \ref{pvVah}, with Remark \ref{invV} or Equation \eqref{parabil}, we deduce that the matrix $\binom{1\ \ \xi}{0\ \ 1}$ is in $\widetilde{\operatorname{V}}(V,q)$ for every $\xi\in\mathbb{F} \oplus V$. As the diagonal matrix $\binom{a\ \ 0}{0\ \ 1}$ is also in $\widetilde{\operatorname{V}}(V,q)$, and these matrices form a subgroup of $\widetilde{\operatorname{V}}(V,q)$ that is isomorphic to $\mathbb{F}^{\times}\rtimes(\mathbb{F} \oplus V,+)$, we can prove the following analogue of Lemma \ref{FxVc}.
\begin{lem}
For $(V,q)$ and $c$, the following are equivalent: (1) There is $\xi\in\mathbb{F} \oplus V$ with $q_{\mathbb{F}}(\xi)=c$. (2) The containments $\widetilde{\mathbf{H}}_{V,q}^{c,o}\subseteq\widetilde{\mathbf{H}}_{V,q}^{c}$ and $\widetilde{\mathbf{K}}_{V,q}^{c,o}\subseteq\widetilde{\mathbf{K}}_{V,q}^{c}$ are strict. (3) The group $\mathbb{F}^{\times}\rtimes(\mathbb{F} \oplus V,+)$ acts on $\widetilde{\mathbf{H}}_{V,q}^{c}$ with more than one orbit. (4) Some elements of $\widetilde{\mathbf{H}}_{V,q}^{c}$ have stabilizers under $\mathbb{F}^{\times}\rtimes(\mathbb{F} \oplus V,+)$. (5) Given $\tau\in\widetilde{\mathbf{H}}_{V,q}^{c,o}$, there is $\binom{\alpha\ \ \beta}{\gamma\ \ \delta}\in\widetilde{\operatorname{V}}(V,q)$ for which the expression $N(\gamma\tau+\delta)$ from Lemma \ref{Nctau+d} is 0. \label{FxVFc}
\end{lem}

\begin{proof}
As in the proof of Lemma \ref{FxVc}, Definition \ref{bdpara} requires the existence of paravectors as in Condition (1), the group $\mathbb{F}^{\times}\rtimes(\mathbb{F} \oplus V,+)$ acts on $\widetilde{\mathbf{K}}_{V,q}^{c,o}$ transitively and with no stabilizers, and if there is $(\infty\lambda)_{b}+\infty\sigma_{c}\in\partial(\mathbb{F} \oplus V_{\mathbb{F}}^{c})$, then the action of $\binom{a\ \ \xi}{0\ \ 1}$ sends it to $(\infty\lambda)_{ab-(\lambda,\xi)_{\mathbb{F}}}+\infty\sigma_{c}$, and there are non-trivial stabilizers. Thus Conditions (1), (2), (3), and (4) are equivalent, and using Proposition \ref{denomvan} and the matrix $\binom{\alpha\ \ \beta}{\gamma\ \ \delta}\in\widetilde{\operatorname{V}}(V,q)$ with $\alpha=0$, $\gamma=-\beta=1$, and $\delta=t\lambda-\xi$ for $\tau=\xi+t\sigma_{c}$ when $q_{\mathbb{F}}(\lambda)=c$ and thus $q_{\mathbb{F},\mathbb{F}}^{c}(t\lambda+t\sigma_{c})=0$ yield the equivalence of Condition (5) as well. This proves the lemma.
\end{proof}

The action of $\widetilde{\operatorname{V}}(V,q)$ have the following properties, analogous to those from Proposition \ref{transact}
\begin{prop}
$\widetilde{\operatorname{V}}(V,q)$ acts transitively on $\widetilde{\mathbf{H}}_{V,q}^{c}$, and the base point $\sigma_{c}$ is stabilized precisely by the matrices $\binom{\delta'\ \ -c\gamma'}{\gamma\ \ \ \ \ \delta\ }$ in which $\gamma$ and $\delta$ are in $\widetilde{\mathcal{T}}(V,q)$ and satisfy $\gamma\delta^{*}\in\mathbb{F} \oplus V$ and $N(\delta)+cN(\gamma)\neq0$. Restricting the action to $\widetilde{\operatorname{SV}}(V,q)$ yields a similar stabilizer but with the equality $N(\delta)+cN(\gamma)=1$, and the action remains transitive in case the conditions from Lemma \ref{FxVFc} are satisfied. \label{transpara}
\end{prop}

\begin{proof}
The subset $\widetilde{\mathbf{H}_{V,q}^{c,o}}$ is contained in a single orbit by Lemma \ref{FxVFc}, and in case this is not the entire space $\widetilde{\mathbf{H}_{V,q}^{c,o}}$, for $\tau=(\infty\lambda)_{b}+\infty\sigma_{c}$ in the complement $\partial(\mathbb{F} \oplus V_{\mathbb{F}}^{c})$, in which $b\neq0$ in case $\mu \in V^{\perp}$, we can take $\delta\in\mathbb{F} \oplus V$ such that $b-(\lambda,\delta)_{\mathbb{F}}\neq0$. Then completing $\delta$ to a matrix in $\widetilde{\operatorname{V}}(V,q)$ with $\alpha=0$ and $\gamma=-\beta=1$, so that $N(\gamma\tau+\delta)_{*}$ from Lemma \ref{inflin} equals that non-vanishing expression and all the elements of $\partial(\mathbb{F} \oplus V_{\mathbb{F}}^{c})$ are also contained in that orbit. For a matrix $\binom{\alpha\ \ \beta}{\gamma\ \ \delta}\in\widetilde{\operatorname{V}}(V,q)$ we again obtain an equivalence between stabilizing $\sigma_{c}$ and the equalities $\alpha=\delta'$ and $\beta=-c\gamma'$, after which the determinant from Corollary \ref{SVpara} again becomes $N(\delta)+cN(\gamma)$. The decomposition of $\widetilde{\operatorname{V}}(V,q)$ using $\widetilde{\operatorname{SV}}(V,q)$ and the matrices $\binom{a\ \ 0}{0\ \ 1}$ with $a\in\mathbb{F}^{\times}$, the transitivity of the former subgroup depends on all the matrices $a\sigma_{c}$ with such $a$ being related via this subgroup, and again, for fixed $a$, the conditions from Lemma \ref{FxVFc} combine with the corollary to Proposition 3 in Chapter IV of \cite{[S]} to yield an element $\xi+d\sigma\in\widetilde{\mathbf{H}}_{V,q}^{c,o}\subseteq\mathbb{F} \oplus V_{\mathbb{F}}^{c}$ with $q_{\mathbb{F},\mathbb{F}}^{c}$-value $-\frac{1}{a}$. Then the matrix $\binom{\alpha\ \ \beta}{\gamma\ \ \delta}$ having entries $\alpha=-a\xi$, $\beta=-\frac{1+aq_{\mathbb{F}}(\xi)}{d}=acd$, $\gamma=d$, and $\delta=\xi$ is in $\widetilde{\operatorname{SV}}(V,q)$ and sends $\sigma_{c}$ to $a\sigma_{c}$ inside $\widetilde{\mathbf{H}}_{V,q}^{c}$. Thus the action of $\widetilde{\operatorname{SV}}(V,q)$ on $\widetilde{\mathbf{H}}_{V,q}^{c}$ is also transitive, and the stabilizer of $\sigma_{c}$ is again determined by the determinant 1 condition inside the stabilizer in $\widetilde{\operatorname{V}}(V,q)$. This proves the proposition.
\end{proof}

The orbits of $\widetilde{\operatorname{SV}}(V,q)$ in $\widetilde{\mathbf{H}}_{V,q}^{c}$ in case the conditions from Lemma \ref{FxVc} do not hold are obtained from the proof of Proposition \ref{transpara} in the following analogue of Corollary \ref{SVorbits}
\begin{cor}
The norms $N(\gamma\tau+\delta)$ with $\binom{\alpha\ \ \beta}{\gamma\ \ \delta}$ from $\widetilde{\operatorname{SV}}(V,q)$ or from $\widetilde{\operatorname{V}}(V,q)$ and from $\tau\in\widetilde{\mathbf{H}}_{V,q}^{c}$ form, in case $\partial(\mathbb{F} \oplus V_{\mathbb{F}}^{c})$ is empty, a subgroup of $\mathbb{F}^{\times}$ that contains the squares such that $\operatorname{SV}(V,q)$ acts on orbits in $\widetilde{\mathbf{H}}_{V,q}^{c}$ that correspond to cosets of this subgroup inside $\mathbb{F}^{\times}$. \label{orbitsSV}
\end{cor}

\begin{proof}
Lemma \ref{FxVFc} yields the non-vanishing of $N(\gamma\tau+\delta)$, the matrices $\binom{a\ \ 0}{0\ \ 1}$ from the left eliminate the difference between $\widetilde{\operatorname{SV}}(V,q)$ and $\widetilde{\operatorname{V}}(V,q)$ for these norms, the translations $\binom{1\ \ \xi`}{0\ \ 1}\in\widetilde{\operatorname{SV}}(V,q)$ restrict attention to the coefficient of $\sigma_{c}$ again, and Proposition \ref{Moebreg} shows that $t\sigma_{c}$ and $s\sigma_{c}$ are related via $\widetilde{\operatorname{SV}}(V,q)$ if and only if they $t/s$ is such a norm. The action shows that these norms are a subgroup of $\mathbb{F}^{\times}$, the matrices $\binom{1/d\ \ 0}{\ 0\ \ \ d}\in\widetilde{\operatorname{SV}}(V,q)$ yield the squares there, and the relation between cosets and $\widetilde{\operatorname{SV}}(V,q)$ is also clear. This proves the corollary.
\end{proof}

\smallskip

The theorem analogous to Theorem \ref{thmVqc} for the paravector case is the following one.
\begin{thm}
For a quadratic space $(V,q)$ over $\mathbb{F}$, and $c\in\mathbb{F}$, the complement $\widetilde{\mathbf{H}}_{V,q}^{c,o}$ of $\mathbb{F} \oplus V$ inside $\mathbb{F} \oplus V_{\mathbb{F}}^{c}$ can be completed to a space $\widetilde{\mathbf{H}}_{V,q}^{c}$ such that the paravector Vahlen group $\widetilde{\operatorname{V}}(V,q)$, as defined in Theorem \ref{pvVah}, acts transitively on this completing by M\"{o}bius transformations. If $V_{U,\mathbb{F}}$ is the direct sum of $V$ with a hyperbolic plane and a line generated by a norm $-1$ vector, then one can identify $\widetilde{\mathbf{H}}_{V,q}^{c}$ with $\widetilde{\mathbf{K}}_{V,q}^{c}$, the set of vectors in $V_{U,\mathbb{F}} \setminus V^{\perp}$ with quadratic value $c$. This identification is $\widetilde{\operatorname{V}}(V,q)$-equivariant, and so is the identification with the left coset space of $\big\{\binom{\delta'\ \ -c\gamma'}{\gamma\ \ \ \ \ \delta\ }\;\big|\;\gamma,\delta\in\widetilde{\mathcal{T}}(V,q),\ \gamma\delta^{*} \in V,\ N(\delta)+cN(\gamma)\neq0\big\}$ inside $\widetilde{\operatorname{V}}(V,q)$. The action of the special paravector Vahlen group $\widetilde{\operatorname{SV}}(V,q)$ defined in Corollary \ref{SVpara} is transitive in case $\widetilde{\mathbf{H}}_{V,q}^{c,o}$ is strictly contained in $\widetilde{\mathbf{H}}_{V,q}^{c}$, and when they are equal, they decompose into cosets, inside $\mathbb{F}^{\times}$, of the norm group of the direct sum of $V$ with the line, and the stabilizer there is the determinant 1 subgroup of the subgroup above. \label{thmpara}
\end{thm}
For $c=0$ the stabilizer from Proposition \ref{transpara} and Theorem \ref{thmpara} becomes the semi-direct product $\widetilde{\Gamma}^{\mathbb{F}^{\times}}(V,q)\rtimes(\mathbb{F} \oplus V,+)$, it can again be conjugated, by the same matrix, to a group of upper-triangular matrices, and inside $\widetilde{\operatorname{SV}}(V,q)$ the semi-direct product involves $\widetilde{\Gamma}^{1}(V,q)$ instead of $\widetilde{\Gamma}^{\mathbb{F}^{\times}}(V,q)$. If $c\neq0$ then these stabilizers are stabilizers of a vector with $q_{U,\mathbb{F}}^{c}$-norm $c\neq0$ inside $\Gamma^{\mathbb{F}^{\times}}_{+}(V_{U,\mathbb{F}},q_{U,\mathbb{F}})$ or $\Gamma^{1}_{+}(V_{U,\mathbb{F}},q_{U,\mathbb{F}})$, which are isomorphic to $\Gamma^{\mathbb{F}^{\times}}_{+}(V_{\mathbb{F},\mathbb{F}}^{-c},q_{\mathbb{F},\mathbb{F}}^{-c})$ and to  $\Gamma^{1}_{+}(V_{\mathbb{F},\mathbb{F}}^{-c},q_{\mathbb{F},\mathbb{F}}^{-c})$ respectively.

\smallskip

Here the hyperbolic space is obtained over $\mathbb{F}=\mathbb{R}$ only from negative definite quadratic spaces $(V,q)$, since we need $(V_{\mathbb{F}},q_{\mathbb{F}})$ to also be definite. The Cayley transform yields a ball model, and conjugating the special paravector Vahlen group by an appropriate element of $\mathcal{C}(V_{U,\mathbb{F}},q_{U,\mathbb{F}})_{+}^{\times}$ yields the group denoted by $\operatorname{SD}(V,q)$ in Section 5 of \cite{[EGM]}. Once again obtaining bounded models in our more general setting requires to work only with $\widetilde{\operatorname{SV}}(V,q)$ and some additional assumptions, so we leave its investigation for further research.

\smallskip

We conclude by remarking that the Vahlen and paravector Vahlen groups were defined in \cite{[Mc]} over more general commutative rings. It may be interesting to check whether models like $\mathbf{H}_{V,q}^{c}$ and $\widetilde{\mathbf{H}}_{V,q}^{c}$ can be meaningful not only over fields of characteristic different from 2. Note, however, that the complete results of \cite{[Mc]} assume that the base ring is an integral domain, which can thus be embedded into its field of fractions and all the work can be done there. In addition, as the results of \cite{[Z]} show, the theory of more general rings brings additional complications, so we leave these questions as well for future research.

\medskip

\noindent\textsc{Einstein Institute of Mathematics, the Hebrew University of Jerusalem, Edmund Safra Campus, Jerusalem 91904, Israel}

\noindent E-mail address: zemels@math.huji.ac.il


\begin{thebibliography}{}{}

\bibitem[Ab]{[Ab]} Ab\l{}amowicz, R., \textsc{Structure of Spin Groups Associated with Degenerate Clifford Algebras}, J. Math. Phys., vol. 27 issue 1, 1--7 (1986).
\bibitem[Ah]{[Ah]} Ahlfors, L., \textsc{M\"{o}bius Transformations and Clifford Numbers}, in \emph{Differential Geometry and Complex Analysis}, ed. I. Chavel and H.M. Farkas, Springer--Verlag Berlin Heidelberg New York, 65-73 (1985).
\bibitem[Ba]{[Ba]} Bass, H., \textsc{Clifford Algebras and Spinor Norms over a Commutative Ring}, Amer. J. Math., vol. 96 no. 1, 156--206 (1974).
\bibitem[EGM]{[EGM]} Elstrodt, J., Grunewald, F., Mennicke, J.  \textsc{Vahlen's Group of Clifford Matrices and Spin-Groups}, Math. Z., vol. 196 issue 3, 369--390 (1987).
\bibitem[L]{[L]} Lounesto, P., \textsc{Clifford Algebras and Spinors}, 2nd edn. London Mathematical Society Lecture Note Series 286, Cambridge University Press, Cambridge xi+338pp (2001).
\bibitem[Ma]{[Ma]} Maks, J., \textsc{Modulo (1,1) Periodicity of Clifford Algebras and the Generalized (anti-)Möbius Transformations}, Ph.D. Thesis, Technische Universiteit Delft (1989).
\bibitem[Mc]{[Mc]} McInroy, J., \textsc{Vahlen Groups Defined over Commutative Rings}, Math. Z., vol. 284 issues 3-4, 901--917 (2016).
\bibitem[MH]{[MH]} Milnor, J., Husemoller, D., \textsc{Symmetric Bilinear Forms}, Ergebnisse der Mathematik und ihrer Grenzgebiete 73, Springer–Verlag, 146pp (1973).
\bibitem[O]{[O]} O'Meara, T., \textsc{Introduction to Quadratic Forms}, Grundlehren der Mathematischen Wissenschaften 117, Springer–Verlag, xiii+342pp (1973).
\bibitem[P]{[P]} Porteous, I. R., \textsc{Topological Geometry}, 2nd edition, Cambridge University Press, Cambridge, x+486pp (1981).
\bibitem[S]{[S]} Serre, J.-P., \textsc{A Course in Arithmetic}, Graduate Texts in Mathematics 7, Springer–Verlag, viii+115pp (1973).
\bibitem[V]{[V]} Vahlen, R., \textsc{\"{U}ber Bewegungen und Complexe Zahlen}, Math. Ann., vol. 55, 585--593 (1902).
\bibitem[Z]{[Z]} Zemel, S., \textsc{Clifford Groups of Arbitrary Quadratic Modules over Commutative Rings}, pre-print, https://arxiv.org/abs/2112.05046 (2021).

\end{thebibliography}
\end{document}